\documentclass{article}
\usepackage[margin=1.25in]{geometry}
\usepackage{mathtools,amssymb,amsthm}
\usepackage{graphicx,color}
\usepackage{subcaption}
\usepackage{pifont}
\usepackage{booktabs}
\usepackage[round,sort]{natbib}
\usepackage{float}
\usepackage[boxed]{algorithm2e}
\usepackage[hyphens]{url}
\usepackage{bm}

\usepackage[toc,page]{appendix}
\usepackage{hyperref}
\usepackage[capitalize]{cleveref}
\usepackage{mathtools,amsmath,amsthm, amssymb}
\usepackage{xifthen}
\usepackage{dsfont}

\newcommand{\cC}{\mathcal{C}}

\newcommand{\cN}{\mathcal{N}}

\newcommand{\cP}{\mathcal{P}}

\newcommand{\cR}{\mathcal{R}}

\newcommand{\RR}{\mathbb{R}}

\newcommand*{\kl}[3][]{%
\ifthenelse{\isempty{#1}}{\operatorname{D}(#2\,\|\,#3)}%
{\operatorname{D}(#2\,\|\,#3\mid#1)}%
}

\newcommand*{\triplenorm}[1]{{\left\vert\kern-0.25ex\left\vert\kern-0.25ex\left\vert #1
    \right\vert\kern-0.25ex\right\vert\kern-0.25ex\right\vert}}

\newcommand{\E}{\mathbb{E}}

\DeclareMathOperator{\Tr}{Tr}

\newcommand*{\ep}{\varepsilon}

\newcommand*{\rd}{\mathrm{d}}
\newcommand*{\dd}{\, \rd}

\DeclareMathOperator{\Ent}{Ent}
\newcommand{\KL}[2]{D_\text{KL}(#1 \| #2)}

\newcommand{\R}{\mathbb{R}}
\newcommand{\Rd}{\mathbb{R}^d}
\newcommand{\eps}{\varepsilon}
\newcommand{\pran}[1]{\left(#1\right)}
\newcommand{\brac}[1]{\left[#1\right]}
\newcommand{\sse}{\subseteq}
\renewcommand{\phi}{\varphi}

\newtheorem{theorem}{Theorem}
\newtheorem{corollary}{Corollary}
\newtheorem{proposition}{Proposition}
\newtheorem{lemma}{Lemma}

\DeclareMathOperator{\OT}{\text{OT}_0}
\DeclareMathOperator{\OTep}{\text{OT}_\eps}

\usepackage{bm}
\usepackage{xcolor}
\usepackage{enumitem}

\definecolor{blue_plots}{HTML}{377EB8}
\definecolor{orange_plots}{HTML}{FF7F00}

\begin{document}

\begin{center} {\LARGE{{Debiaser Beware: Pitfalls of Centering Regularized Transport Maps}}}

{\large{

\vspace*{.3in}

\begin{tabular}{cccc}
Aram-Alexandre Pooladian$^*$, Marco Cuturi$^\circ$, Jonathan Niles-Weed$^{*\dagger}$,\\
\end{tabular}
 
{
\vspace*{.1in}
\begin{tabular}{c}
				$^*$Center for Data Science, New York University\\
				$^\dagger$Courant Institute of Mathematical Sciences, New York University \\
 				\texttt{ap6599@nyu.edu, jnw@cims.nyu.edu}\\
 				$^\circ$ Google Research, currently at Apple.\\
 				\texttt{cuturi@apple.com}\\
\end{tabular} 
}

}}
\vspace*{.1in}

\today

\end{center}

\vspace*{.1in}

\begin{abstract}
Estimating optimal transport (OT) maps (\emph{a.k.a.} Monge maps) between two measures $P$ and $Q$ is a problem fraught with computational and statistical challenges. A promising approach lies in using the dual potential functions obtained when solving an entropy-regularized OT problem between samples $P_n$ and $Q_n$, which can be used to recover an approximately optimal map. The negentropy penalization in that scheme introduces, however, an estimation bias that grows with the regularization strength. A well-known remedy to debias such estimates, which has gained wide popularity among practitioners of regularized OT, is to center them, by subtracting auxiliary  problems involving $P_n$ and itself, as well as $Q_n$ and itself. We do prove that, under favorable conditions on $P$ and $Q$, debiasing can yield better approximations to the Monge map.
However, and perhaps surprisingly, we present a few cases in which debiasing is provably detrimental in a statistical sense, notably when the regularization strength is large or the number of samples is small. These claims are validated experimentally on synthetic and real datasets, and should reopen the debate on whether debiasing is needed when using entropic optimal transport.
\end{abstract}

\section{Introduction}\label{sec:intro}
Estimating an optimal transport (OT) map from a source measure $P$ to a target measure $Q$ is an increasingly central issue in machine learning: such a map, when successfully learned from data, would in principle allow the generation of new samples from $Q$ by pushing particles sampled from $P$.
For example, if we are given full access to the source measure (typically a Gaussian) but only samples from the target, this problem can be viewed with the lens of normalizing flows~\citep{grathwohl2018ffjord, huang2021convex,finlay2020learning} with tight connections with the problem of estimating generative models~\citep{WassersteinGAN,2017-Genevay-AutoDiff,salimans2018improving}.
When both $P$ and $Q$ can only be accessed through i.i.d samples $X_1,\ldots,X_n \sim P$ and $Y_1,\ldots,Y_n \sim Q$, estimating such maps is even more challenging, yet increasingly relevant when, for instance, trying to infer cellular evolution from population measurements~\citep{schiebinger2019optimal,moriel2021novosparc,Demetci2021.SCOTv2,dai2018autoencoder}.

Formally, we call $T_0 : \Rd \to \Rd$ an OT map, or a Monge map, when it is the minimizer to the~\citeauthor{Monge1781} problem \citeyearpar{Monge1781}:
\begin{align}\label{eq: monge_p}
\min_{T \in \mathcal{T}(P,Q)} \int \frac{1}{2}\| x - T(x)\|^2_2 \dd P(x),
\end{align}
where $P$ and $Q$ are two probability measures on $\Omega \sse \R^d$, and $\mathcal{T}(P,Q)$ is the set of \textit{pushforward} maps from $P$ to $Q$, namely  $\mathcal{T}(P,Q) := \{T : \Omega \to \Omega \ | \ T_\sharp P := P \circ T^{-1} = Q\}$. The existence of such a minimizer is guaranteed under mild regularity conditions on the two measures \citep{santambrogio2015optimal}. The task of map estimation is to provide an estimator $\hat{T}_n$ on the basis of samples $P_n$ from $P$ and $Q_n$ from $Q$ whose expected $L^2(P)$ distance to $T_0$ is small.

When $T_0 \in \cC^\alpha$ (the space of $\lfloor \alpha \rfloor$-times differentiable functions whose $\lfloor \alpha \rfloor$th derivative is $\alpha - \lfloor \alpha \rfloor$ H{\"o}lder smooth) and $P$ and $Q$ satisfy additional technical assumptions, \citet{hutter2021minimax} provided the first estimator that achieved the following estimation rate, 
\begin{align}\label{hr_bound}
\E \|\hat{T}_n - T_0\|^2_{L^2(P)} \lesssim n^{-\frac{2\alpha}{2\alpha - 2 + d}}\log^3(n) \,,
\end{align}
which they showed to be minimax optimal up to logarithmic factors. Though statistically optimal, their estimator requires gridding the space of observations, with computational complexity scaling exponentially in $d$. This work left open the question of finding a computationally tractable estimator that has good statistical properties. While \citet{manole2021plugin} and \citet{deb2021rates} proposed minimax optimal estimators, the fastest estimator between these two works has a runtime of $\Tilde{O}(n^3)$. Similarly, \citet{muzellec2021near} devised a nearly minimax map estimator in the high smoothness regime based on kernel sum-of-squares, however the underlying constant depends exponentially in $d$ as their regularization parameter goes to zero.

Leveraging the computational benefits of the Sinkhorn algorithm, \citet{pooladian2021entropic} analyzed the computational and statistical properties of the \textit{Entropic map}, which to any input $x$ associates the conditional expectation of $Y$ given $X=x$ under the entropy-regularized optimal coupling. This estimator can be computed in $\Tilde{O}(n^2)$ time and is near optimal when the regularity of $T_0^{-1}$ is low.
Together, these results indicate that entropic regularization is a promising methodology for the estimation of Monge maps.

However, entropic regularization is known to introduce some form of bias when solving OT problems, notably when the regularization parameter is large \citep{cuturi2018semidual,schmitzer2019stabilized}.
A popular remedy for this phenomenon has been to \textit{debias} said estimates by introducing auxiliary OT problems \citep{GenPeyCut18,feydy2019interpolating,chizat2020faster}. 
This debiasing procedure has been shown to offer both theoretical and practical benefits for estimating transport distances, suggesting that debiasing should be applied systematically when considering regularized OT problems, notably with the entropic estimation of Monge maps.
Surprisingly, we show that the situation is more nuanced: we prove that debiasing can offer benefits when the regularization parameter is small and the number of samples is large but can substantially \emph{degrade} statistical performance when these conditions are not met.

\textbf{Main Contributions:} In this paper, we explore whether debiasing is a fruitful approach for estimating Monge maps. Specifically, our contributions are the following:
\begin{enumerate}
    \item We define the debiased map estimator, and prove a rate of convergence to the Monge map in the small regularization limit (§~\ref{sec:results}, \cref{thm: eps_d_small});
    \item We prove quantitative differences between the Entropic map and its debiased variant in the large regularization regime, and present examples where debiasing provably degrades performance (§~\ref{sec:results}, \Cref{thm: err_diff_main} and \Cref{thm: counter_ex});
    \item We derive closed-form expressions in the Gaussian-to-Gaussian case, and show that debiasing yields a better approximation to the Monge map when the regularization parameter is small (§~\ref{sec:gauss});
    \item We illustrate advantages and pitfalls of debiasing on several numerical experiments (§~\ref{sec:exp}). Notably, we show that finite-sample effects can substantially worsen the performance of the debiased map.
\end{enumerate}
We use the symbol $C$ to denote a positive constant whose value may change from line to line, and write $a\lesssim b$ if there exists a constant $C > 0$ such that $a \leq Cb$. 
\section{Background on Optimal transport}\label{sec:background}

\subsection{Optimal transport without regularization}\label{sec: background_ot}
For $\Omega \sse \Rd$ compact, we define $\cP(\Omega)$ to be the space of (Borel) probability measures with support contained in $\Omega$, and~$\cP_{ac}(\Omega)$ be those with densities. We first recall \citeauthor{Bre91}'s fundamental result on the existence of an optimal map.
\begin{theorem}[\citeauthor{Bre91}'s Theorem \citeyearpar{Bre91}]\label{thm: brenier_thm}
Let $P \in \cP_{ac}(\Omega)$ and $Q \in \cP(\Omega)$. Then there exists a solution $T_0$ to~\Cref{eq: monge_p}, with $T_0 = \text{Id} - \nabla f_0$, where $f_0$ is a 1-strongly concave function solving
\begin{align}\label{eq: ot_dual}
\sup_{(f,g)\in\Phi}\int f \dd P + \int g \dd Q \,,
\end{align}
where $ \Phi = \{(f, g) \in L^1(P) \times L^1(Q):  f(x) + g(y) \leq \frac{1}{2}\|x - y\|_2^2 \, \, (x,y \in \Omega)\}$.
\end{theorem}
In other words, Brenier's theorem asserts that when $P$ has a density function, the optimal map is the gradient of a convex function, $x\mapsto \tfrac{1}{2}\|x\|^2-f_0(x)$. We call the maximizers $(f_0,g_0)$ to \Cref{eq: ot_dual} \textit{optimal (Kantorovich) potentials}.

If $P$ does not have a density, an optimal transport map may not exist between $P$ and $Q$. To remedy this, \citet{Kan42} devised a convex relaxation of the optimal transport problem where one is optimizing over \textit{plans} instead of maps
\begin{align}\label{eq: kant_p}
\OT(P,Q) := \inf_{\pi \in \Pi(P,Q)} \iint \frac{1}{2}\| x - y\|_2^2 \dd \pi(x,y),
\end{align}
where $ \Pi(P,Q)$ is the set of joint probability measures with marginals $P$ and $Q$, called the set of \textit{couplings}.
When an optimal map exists, it also gives rise to a solution to \cref{eq: kant_p}; however, 
unlike the Monge problem, \cref{eq: kant_p} always admits a minimizer when $P$ and $Q$ have finite second moments. We call such a minimizer an \textit{optimal plan}, denoted $\pi_0$.
Note that $\OT$ is also the definition of $\frac{1}{2}W_2^2(P,Q)$, the squared 2-Wasserstein distance between $P$ and $Q$.

\subsection{Entropic optimal transport}
We define \textit{Entropic optimal transport} to be the objective function that arises when we add the \textit{Kullback--Liebler (KL) divergence}, denoted $D_\text{KL}(\cdot\|\cdot)$, as a regularizer to \cref{eq: kant_p}. Recall that $\KL \mu \nu = \int \log( \frac {\dd \mu} {\dd \nu} )\dd \mu$. Thus, for $P,Q \in \cP_{\text{ac}}(\Omega)$, we write the $\eps$-regularized OT problem as
\begin{align}\label{eq: eot_p}
    \OTep(P,Q) &= \inf_{\pi \in \Gamma(P,Q)} \iint c \dd \pi + \eps \KL{\pi}{P\otimes Q}\\
    &= \inf_{\pi \in \Gamma(P,Q)} \iint c \dd \pi + \eps\iint \log \pi \dd \pi -\eps (\Ent(P) +\Ent(Q))\,,\nonumber
\end{align}
where we abbreviate $c(x, y) = \frac 12 \|x - y\|_2^2$, and the \textit{entropy} of an absolutely continuous probability measure is denoted by $\text{Ent}(P) = \int p(x) \log(p(x)) \dd x$; note that the first line also holds in the absence of densities. Since $\OTep$ is now strongly convex, we call its unique minimizer the \textit{optimal entropic plan}, denoted $\pi_\eps$, which exists under a finite-second moment condition~\citep{PeyCut19}.

Similar in spirit to \cref{eq: ot_dual}, the Entropic OT problem admits the following dual formulation under mild conditions,
\begin{multline}\label{eq: eot_d1}
    \OTep(P,Q) = \!\!\!\!\!\!\!\sup_{f \in L^1(P), g\in L^1(Q)} \int f \dd P + \int g \dd Q 
    -\eps \iint e^{(f(x)+g(y)-c(x,y))/\eps}\dd P(x)\dd Q(y) + \eps \,,
\end{multline}
where the maximizers, denoted $(f_\eps,g_\eps)$ are called \textit{optimal entropic potentials}. Primal-dual relationships provide that the optimal entropic coupling is closely related to $(f_\eps,g_\eps)$ through the following representation \citep{Csi75}:
\begin{align*}
    \dd \pi_\eps(x,y) = \exp\pran{\frac{f_\eps(x) + g_\eps(y) - c(x,y)}{\eps}}\dd P(x) \dd Q(y)\,.
\end{align*}
As a result, and since $\pi_\eps$ integrates to $1$, it holds for the maximizers that \Cref{eq: eot_d1} simplifies to the expression (see also other facts in \Cref{sec: main_proofs_ot}):
\begin{equation*}
    \OTep(P,Q) = \int f_\eps \dd P + \int g_\eps \dd Q\,.
\end{equation*}

\subsection{The Sinkhorn divergence}\label{subsec:sinkhorndiv}
While $\OTep(P,Q)$ is a reasonable approximation of $W_2^2(P,Q)$, we have that $\OTep(P,P) \neq 0$ but $W_2^2(P,P) = 0$, which leads to an inherent bias. In view of this phenomenon, several works have suggested \textit{centering} the Entropic OT objective in the following manner 
\begin{align}\label{eq: sinkdiv}
S_\eps(P,Q):=\OTep(P,Q)-\tfrac12\OTep(P,P)-\tfrac12\OTep(Q,Q),
\end{align}
resulting in the so-called \textit{Sinkhorn divergence}. Centering the Entropic OT objective is akin to \textit{debiasing} the objective, as now $S_\eps(P,Q) = 0 \iff P=Q $ \citep{feydy2019interpolating}. 

The corrective terms, $\OTep(P,P)$ and $\OTep(Q,Q)$, each admit their own dual formulation that is similar to \Cref{eq: eot_d1}, with the key difference that they each possess only one optimal potential \citep{feydy2019interpolating}. For example, we denote by $\alpha_\eps$ the optimal entropic \textit{self-potential}:
\begin{align}
\OTep(P,P) &= \sup_{\alpha \in L^1(P)} \int 2\alpha\dd P \quad - \eps \iint \pran{e^{(\alpha(x)+\alpha(y)-c(x,y))/\eps}  - 1}\dd P(x)\dd P(y) = 2\int \alpha_\eps \dd P\,, \nonumber
\end{align}
and similarly by $\beta_\eps$ the optimal self-potential corresponding to $\OTep(Q,Q)$. Thus, at optimality, one has the relationship
\begin{align}
    S_\eps(P,Q) &= \int (f_\eps - \alpha_\ep) \dd P + \int (g_\eps - \beta_\eps)\dd Q \\
    &=: \int f_\eps^D \dd P + \int g_\eps^D \dd Q \,,
\end{align}
where we define $(f_\eps^D,g_\eps^D)$ to be the \textit{debiased optimal entropic potentials}, or \textit{Sinkhorn potentials}.

The Sinkhorn divergence has found many applications in the machine learning community as a way of eliminating the ``bias" in the entropic OT objective \citep{GenPeyCut18,feydy2019interpolating,chizat2020faster}. For example, in the case of estimating the Wasserstein distance, this debiasing effect can be made rigorous: \citet{pal2019difference} showed that $\OTep$ gives an additive $O(\eps\log(\eps^{-1}))$ approximation to $W_2^2$. By contrast, \citet{chizat2020faster} showed that $S_\eps = W_2^2 + O(\eps^2)$.

\subsection{The Entropic map}\label{subsec:entropicmap}
\citet{pooladian2021entropic} studied the \textit{Entropic map}, defined to be the \textit{barycentric projection} of the optimal entropic plan, defined as
\begin{equation}\label{eq: t_eps_expectation}
    T_\eps(x) := \E_{\pi_\eps}[Y|X=x],
\end{equation}
as an estimator for $T_0$ between two measures $P$ and $Q$. This object is partially motivated by the following observation, which acts as an entropic analogue to Brenier's Theorem \citep[see][Prop. 2]{pooladian2021entropic}:
\begin{align}
    T_\eps = \text{Id} - \nabla f_\eps\,.
\end{align}
Moreover, in the finite-sample regime, this estimator can be computed in $\Tilde{O}(n^2)$ operations and can be efficiently parallelized on GPUs, which is in contrast to other estimators in the literature that require either at least $\Tilde{O}(n^3)$ complexity and are not easily parallelizeable \citep{manole2021plugin,deb2021rates,hutter2021minimax}.

\section{Entropic map vs. Sinkhorn map}\label{sec:results}

In this work, we introduce the \textit{Sinkhorn map}
\begin{equation}\label{eq: sinkhorn_map}
    T_\eps^D(x) := x - \nabla f_\eps^D(x)
    = T_\eps(x) + \nabla \alpha_\eps(x)\,
\end{equation}
which can be viewed as the natural, and perhaps desirable, extension of the Entropic map from §~\ref{subsec:entropicmap} following the centering or \textit{debiasing} principle advocated in §~\ref{subsec:sinkhorndiv}.




\subsection{Biased and Debiased MSE}
As discussed in §~\ref{subsec:sinkhorndiv}, debiasing yields crucial improvements in approximating $W_2^2(P,Q)$. Our primary focus is to determine whether or not such an improvement can be made in the context of map estimation. Assuming we have full access to the source and target measures $P$ and $Q$, we are interested in comparing the limiting behaviors of their Mean-Squared Error (MSE), $\cR(T_\eps)$ and $\cR(T_\eps^D)$, where for a given $T_0$ and any estimator $S:\mathbb{R}^d\rightarrow\mathbb{R}^d$,
\begin{equation}
\cR(S) := \|S - T_0\|^2_{L^2(P)}\,.
\end{equation}

\paragraph{Small \bm{$\eps$} limit}
As a by-product of their analysis, \citet{pooladian2021entropic} proved the following estimate:
\begin{proposition}\label{prop: pooladian_main}
Let $P,Q \in \cP_{\text{ac}}(\Omega)$ with bounded densities with compact support, with the density of $Q$ bounded from below. Let $T_0 := \nabla \phi_0$ be the optimal map between $P$ and $Q$, where $\mu I\preceq \nabla^2 \phi_0 \preceq LI$, and $T_0^{-1} \in \cC^{\alpha}$ for $\alpha>1$. Then for $\eps$ sufficiently small,
\begin{align*}
    \cR(T_\eps) \lesssim \eps^2I_0(P,Q) + \eps^{(\Bar{\alpha}+1)/2}\,,
\end{align*}
where $\bar{\alpha} = 3\wedge\alpha$ and $I_0(P,Q)$ is the integrated Fisher information along the $W_2^2(P,Q)$ geodesic.
\end{proposition}
We notice the following phenomena: in the high smoothness regime, we obtain a quadratic rate of convergence. On the other hand, as $\alpha\to 1$, the convergence rate is linear in the regularization parameter.

Our first contribution, which is potentially of independent interest, is a gradient estimate of the self-potential $\alpha_\eps$: we show that the effect of adding $\nabla \alpha_\eps$ to the Entropic map as defined in \cref{eq: sinkhorn_map} is small. Proofs of the results found in this section are deferred to \Cref{sec: main_proofs}.
\begin{lemma}\label{lem: alpha_eps_bound}
Let $\alpha_\eps$ be the solution to $\OTep(P,P)$ with $P$ having a density function with compact support. Then
\begin{align}
\|\nabla \alpha_\eps\|^2_{L^2(P)} \leq \frac{\eps^2}{4} I_0(P,P)\,,
\end{align}
where $I_0(P,P) := \int \|\nabla \log p(x)\|^2 p(x) \dd x$ is the Fisher information of $P$.
\end{lemma}
Combining \cref{prop: pooladian_main} with \cref{lem: alpha_eps_bound}, we obtain the following result, which shows that in the $\eps \to 0$ regime,  we are no worse when we debias our map estimator. 
\begin{theorem}\label{thm: eps_d_small}
Under the assumptions of \Cref{prop: pooladian_main}, for $\eps$ small enough, there exists $C_1,C_2 > 0$ such that
$$\cR(T^D_\eps) \leq \cR(T_\eps) + C_1\eps^2 + C_2 \eps^{(\bar{\alpha}+5)/4}\,,$$
where $C_1$ and $C_2$ are constants depending on $I_0(P,P)$ and $I_0(P,Q)$.
In particular, if $I_0(P, P)$ and $I_0(P, Q)$ are finite, then as $\eps \to 0$,
\begin{equation*}
    T_\eps \to T_0, \quad T_\eps^D \to T_0 \quad \text{in $L^2(P)$.}
\end{equation*}
\end{theorem}

\paragraph{Large \bm{$\eps$} limit}
We now turn our attention to the large $\eps$ regime. \Cref{prop: map_conv} characterizes the limiting behavior of the maps, whereas \Cref{thm: err_diff_main} concerns their MSE.
\begin{proposition}\label{prop: map_conv}
Let $P, Q \in \cP(\Omega)$, and let $\mu_P, \mu_Q \in \RR^d$ denote their means.
As $\eps \to \infty$, and in $L^2(P)$ sense, one has
\begin{equation}
T_\eps \  {\rightarrow} \ \mu_Q ,\quad T_\eps^D \   {\rightarrow} \ \text{Id} + (\mu_Q - \mu_P)\,.
\end{equation}
\end{proposition}
\begin{figure}[h!]
    \centering
    \includegraphics[width=0.55\textwidth]{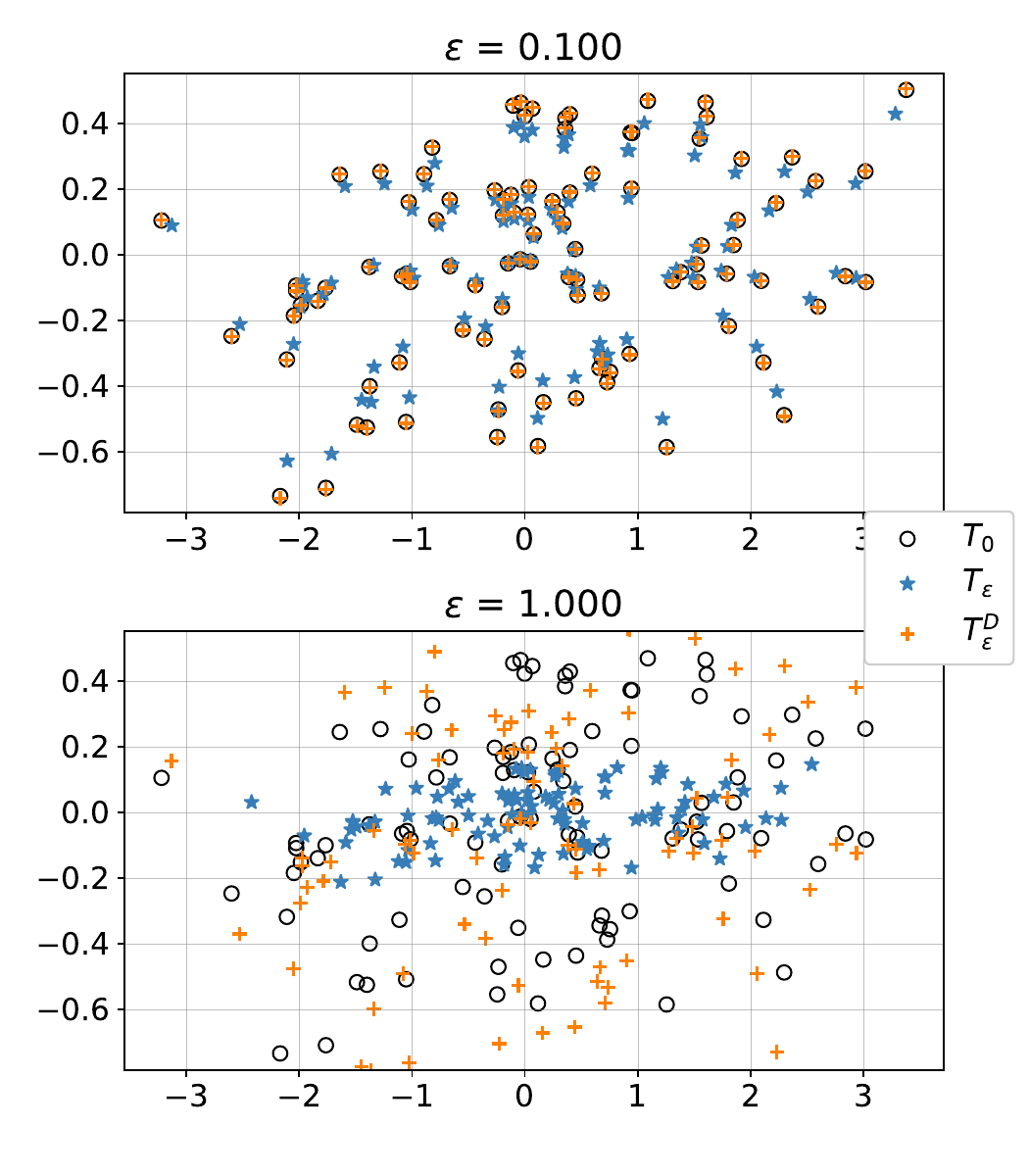}
    \caption{Visual behavior of the Entropic ${T}_\eps$ and Sinkhorn (debiased) ${T}_\eps^D$ maps between two Gaussians using closed-form expressions derived in \S~\ref{sec:gauss}, evaluated on arbitrary samples from the input Gaussian. The estimators approach $T_0$ as $\varepsilon\approx0$. When $\eps$ is relatively large, we see that ${T}_\eps$ is biased towards the mean of the target, which in this case is zero. In contrast, ${T}_\eps^D$ is still maintaining the overall shape of the target.}
    \label{fig: viz_elliptic}
\end{figure}

\Cref{prop: map_conv} shows that as $\eps \to \infty$, the Entropic map approaches the constant map $x \mapsto \mu_Q$, whereas the Sinkhorn map approaches the linear map $x \mapsto x + (\mu_Q - \mu_P)$. This contrasts with the $\ep \to 0$ limit, where both maps tend to $T_0$; these two behaviors are illustrated in \Cref{fig: viz_elliptic}.

\Cref{prop: map_conv} implies the following expressions for the limiting values of the MSE.
\begin{theorem}\label{thm: err_diff_main}
For $P,Q \in \cP(\Omega)$, as $\eps \to \infty$, $$\cR(T_\eps) \to \mathrm{Var}(Q) , \quad \cR(T^D_\eps) \to W^2_2(\bar{P},\bar{Q}),$$ where $\bar{P} := P - \mu_P$ and $\bar{Q} := Q - \mu_Q.$
\end{theorem}

Crucially, neither $\mathrm{Var}(Q)$ nor $W^2_2(\bar{P},\bar{Q})$ dominates the other in general, and whether $\cR(T_\eps)$ or $\cR(T^D_\eps)$ is larger depends on the properties of $P$ and $Q$.
\Cref{thm: err_diff_main} exposes one potential pitfall of debiasing: it can either help or harm the estimate of the Monge map, depending on the size of $\eps$ and the properties of $P$ and $Q$.

\Cref{fig:closedform_behavior} illustrates this effect in the Gaussian-to-Gaussian setting. 
When $\eps$ is small, debiasing yields better estimates, but when $\eps$ is large, the approximation error of the maps converge to the proven quantities in \Cref{thm: err_diff_main}, and debiasing does not necessarily dominate the biased estimator.

\paragraph{Debiasing can be harmful even for small \bm{$\eps$}}
\Cref{thm: err_diff_main} shows that debiasing is not uniformly helpful when $\eps$ is large.
Nevertheless, one may conjecture that this phenomenon only manifests when $\eps \to \infty$, and that there exists $\eps_0$ such that debiasing is benign for $\eps \leq \eps_0$.
Our next result shows that this conjecture is emphatically false: it is possible for debiasing to yield \emph{arbitrarily worse} estimates of the Monge map, no matter how small $\eps$ is taken.
\begin{theorem}\label{thm: counter_ex}
For any $\eps < 1$ and any $M > 0$, there exist a pair of densities $(P_\eps, Q_\eps)$ for which
$\cR(T^D_\eps) \geq M \cR(T_\eps)$.
\end{theorem}

\begin{figure*}[h!]
    \centering
    \includegraphics[width=\textwidth]{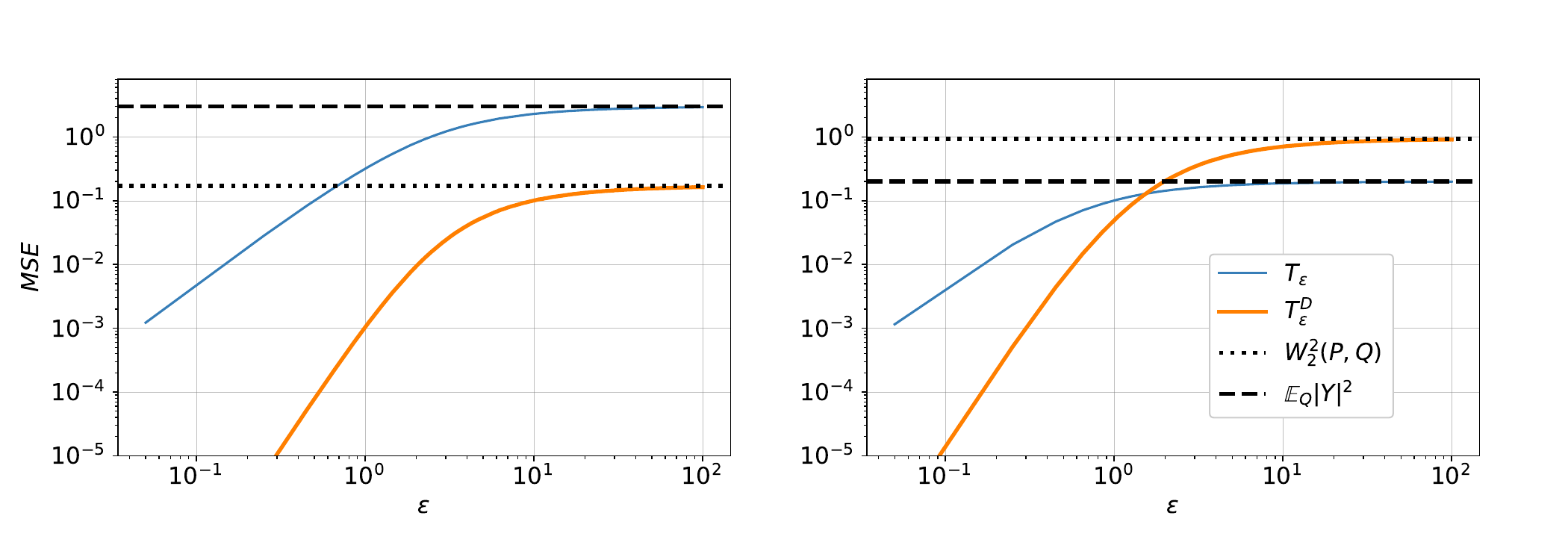}
    \caption{Validation of theoretical results from \Cref{sec:results} and \Cref{sec:gauss}, when $P = \cN(0,I_2)$ and $Q_1 = \cN(0,\text{diag}(2,1))$ (left figure) and $Q_2 = \cN(0,\text{diag}(0.1,0.1))$ (right figure). Making use of closed-form expressions, Monte Carlo integration was used to make the figures.}
    \label{fig:closedform_behavior}
\end{figure*}

\section{Case Study: Transport between Gaussians}\label{sec:gauss}

We now explicitly analyze the case $P = \cN(0,A)$ and $Q = \cN(b,B)$ with $A , B\succ 0$ (the case where the mean of the first measure is non-zero can be recovered with a simple translation). The optimal transport map between two such Gaussians has the following closed-form solution \citep{gelbrich1990formula}
\begin{align*}
    T_0(x) &= C^{AB}_0x + b := A^{-1/2}(A^{1/2}BA^{1/2})^{1/2}A^{-1/2}x + b.
\end{align*} Since the Entropic map is defined as the conditional mean of the optimal entropic plan, this also has a closed form:

\begin{proposition}\label{prop: teps_g}
\citep{janati2020entropic,mallasto2021entropy} For $P = \cN(0,A)$ and $Q = \cN(b,B)$ and $\eps > 0$, the Entropic map is given by $T_\eps(x) = C_\eps^{AB}x + b$, where
\begin{align}
    C_\eps^{AB} := (A^{-1/2}[A^{1/2}BA^{1/2}& + (\eps^2/4)I]^{1/2}A^{-1/2} - (\eps/2)A^{-1} ).
\end{align}
\end{proposition}

We now characterize the Sinkhorn map, and prove that in the small $\eps$ regime, it is a better estimator of $T_0$ than the Entropic map; see \Cref{fig:closedform_behavior}. Proofs pertaining to this section are found in \Cref{sec: gauss_proofs}.

\begin{proposition}\label{prop: alpha_g}
For $P = \cN(0,A)$, $\alpha_\eps(x) = \frac{1}{2}x^\top(I - C^{AA}_\eps)x$, where
\begin{align*}
C_\eps^{AA} = A^{-1/2}(A^2 + (\eps^2/4)I)^{1/2}A^{-1/2} - (\eps/2)A^{-1}\,.
\end{align*}
\end{proposition}

\begin{corollary}\label{cor: tepsd_g}
For $P = \cN(0,A)$ and $Q = \cN(b,B)$, the debiased map estimator is $T_\eps^D(x) := \tilde{C}_\eps^{AB}x + b$, where
\begin{align}
\tilde{C}^{AB}_\eps = &(A^{-1/2}(A^{1/2}BA^{1/2} + (\eps^2/4)I)^{1/2}A^{-1/2} + I - A^{-1/2}(A^2 + (\eps^2/4)I)^{1/2}A^{-1/2}  ). 
\end{align}
\end{corollary}
\begin{proof}
This follows from our definition of $T_\eps^D(x) = T_\eps(x) + \nabla \alpha_\eps(x)$, and Propositions \ref{prop: teps_g} and \ref{prop: alpha_g}, noting that $\tilde{C}_\eps^{AB} = C_{\eps}^{AB}+I - C_{\eps}^{AA}$.
\end{proof}

With these closed-form expressions in hand, we can proceed with the following theorems. 
\begin{theorem}\label{thm: err_g1}
For $P = \cN(0,A)$ and $Q = \cN(b,B)$, 
$$ \|T_\eps - T_0\|_{L^2(P)}^2 \leq \frac{\eps^2}{4}I_0(P,P) + O_{A,B}(\eps^4)\, $$
where $I_0(P,P) = \Tr(A^{-1})$.
\end{theorem}
A key difference between \Cref{thm: err_g1} and \Cref{prop: pooladian_main} is that the leading-order error term here scales like ${\eps^2} I_0(P,P)$ instead of $\eps^2 I_0(P,Q)$. This is made possible by having a more precise control on $T_\eps$. While at a glance this seems like a minor difference, it allows us to improve the convergence rate for the debiased Entropic map. We provide a short proof sketch that outlines why this true; the full proofs for \Cref{thm: err_g1,thm: err_g2} can be found in \Cref{sec: gauss_proofs}. Essentially, \Cref{thm: err_g2} is possible because the leading-order contributions of $T_\eps$ are exactly canceled out by those from $\nabla \alpha_\eps$, resulting in improved estimates for the debiased map. 
\begin{theorem}\label{thm: err_g2}
For $P = \cN(0,A)$ and $Q = \cN(b,B)$, 
$$ \|T_\eps^D - T_0\|_{L^2(P)}^2 \lesssim \eps^4 + O_{A,B}(\eps^6)\,. $$
Thus, for $\eps$ small enough, $\cR(T^D_\eps) \leq \cR(T_\eps)$.
\end{theorem}
\begin{proof}[Proof sketch]
Using a Taylor expansion of the explicit expressions of the maps for the Gaussian-to-Gaussian case, one can check that
\begin{align*}
    T_\eps^D(x) &= T_\eps(x) + \nabla \alpha_\eps(x) \\
    &= \pran{T_0(x) - \frac{\eps}{2}\nabla \log(p(x)) + O(\eps^2)} + \pran{0 + \frac{\eps}{2}\nabla\log(p(x)) + O(\eps^2)}\,. \\
    &= T_0(x) + O(\eps^2).
\end{align*}
Thus, one expects $\|T_\eps^D - T_0\|^2_{L^2(P)}  = O(\eps^4)$.
\end{proof}
A visual representation of these rates appears in \Cref{fig:closedform_behavior}: for $\eps$ small, we observe convergence rates with order $O(\eps^2)$ and $O(\eps^4)$ for the biased and debiased MSE, respectively.

\section{Experiments}\label{sec:exp}
\subsection{Finite-sample map estimation}
The results of \S\S~\ref{sec:results}-\ref{sec:gauss} illustrate the benefits and pitfalls of debiasing to obtain good approximations of the Monge map between two \emph{known} distributions $P$ and $Q$.
This next section shows the extent to which these findings extend to the finite-sample regime. Our findings are double-edged: when the number of samples is sufficiently large and the dimension is moderate, debiasing can still have benefits.
However, for smaller sample sizes and larger dimensions, the statistical error begins to dominate, and debiasing yields \emph{worse} performance.
We illustrate both phenomena on a suite of benchmark examples.

In the finite-sample regime, the Entropic map can be written
\begin{align*}
    \hat{T}_\eps(x) = \frac{\sum_{i=1}^n Y_i e^{(\hat{g}_\eps(Y_i) - c(x,Y_i))/\eps}}{\sum_{i=1}^n e^{(\hat{g}_\eps(Y_i) - c(x,Y_i))/\eps}}\,,
\end{align*}
 where we write $(\hat{f}_\eps,\hat{g}_\eps)$ as the optimal entropic potentials for $\OTep(P_n,Q_n)$. Through the celebrated Sinkhorn's algorithm, it is well-known that these potentials can be computed in $\Tilde{O}(n^2)$ operations \citep{AltWeeRig17,PeyCut19}, and evaluating the map is done in linear time. A similar procedure can be performed for the self-potential $\hat{\alpha}_\eps$ though it empirically converges significantly faster \citep{feydy2019interpolating}. Together, this results in the finite-sample Sinkhorn map estimator
 \begin{align*}
     \hat{T}^D_\eps(x) &= \hat{T}_\eps(x) + \nabla \hat{\alpha}_\eps(x) \\
     &= \hat{T}_\eps(x) + x - \frac{\sum_{i=1}^n X_i e^{(\hat{\alpha}_\eps(X_i) - c(x,X_i))/\eps}}{\sum_{i=1}^n e^{(\hat{\alpha}_\eps(X_i) - c(x,X_i))/\eps}}\,.
 \end{align*}
 
 Our numerical experiments were performed using Google Colab Pro, where our code is adapted from \cite{chizat2020faster} and is publicly available. Across all plots, we compute the MSE via Monte-Carlo integration, where we always use $5\cdot10^5$ points. Unless otherwise specified, we perform our experiments across 20 random trials in order to generate error bars. More details are included in \Cref{sec: app_numerics}. 
\subsubsection{Smooth maps}
Restricting our attention to $P$ admitting a density with compact support, we define our target distribution as $Q := T_\sharp P$ for $T$ smooth.  For a given dimension, we fix our choice of $\eps$, compute the Entropic and Sinkhorn maps as $n$ varies, and then compute their MSE.

A first example is when $P = \text{Unif}([-1,1]^d)$ with the map
\begin{align*}
\textbf{\text{(E1)}} \quad    T(x) = \Omega_dx
\end{align*}
where we use the diagonal matrix example from \cite{paty2020regularity}, defined as $(\Omega_d)_{ii} := 0.8 - \frac{0.4}{d-1}(i-1) \quad (i \in [d]).$
\Cref{fig: d5_examples} shows the MSE of the two estimators as a function of the number of samples, where the Sinkhorn map is better by over an order of magnitude. Due to space constraints, the remainder of our synthetic experiments in this setting are deferred to \Cref{sec: app_numerics}. Though, in all cases, the Sinkhorn map always has a lower MSE than the Entropic map. 

\begin{figure}[h!]
    \centering
    \begin{subfigure}[h]{0.55\textwidth}
        \includegraphics[width=\textwidth]{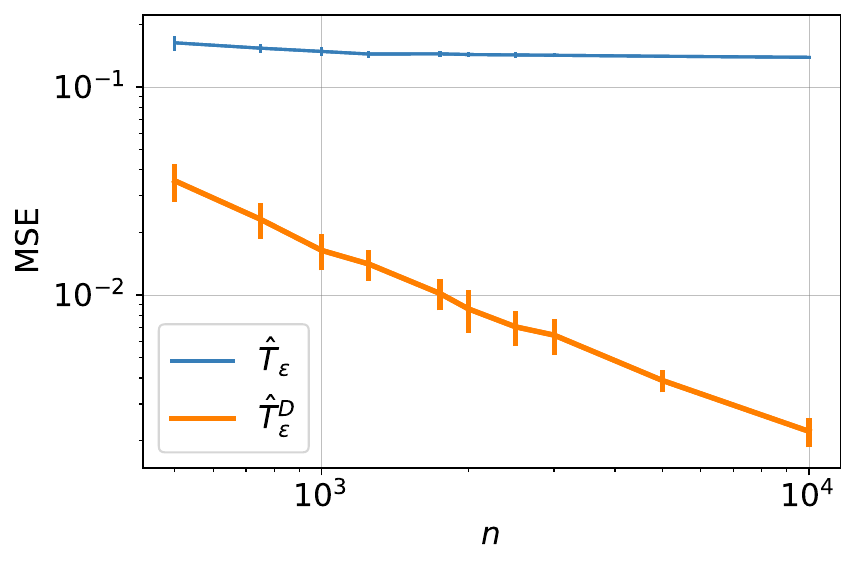}
        \caption{Example \textbf{(E1)} in $d=5$ with $\eps=0.5$}
    \end{subfigure}
    \caption{$\!$For moderate dimension and relatively large $\eps$, \textcolor{orange_plots}{$\hat{T}_\eps^D$} better approximates \textbf{(E1)} compared to \textcolor{blue_plots}{$\hat{T}_\eps$}.}
    \label{fig: d5_examples}
\end{figure}

\subsubsection{Investigating the impact of smoothness}
We briefly investigate the impact of debiasing when the optimal map is itself non-smooth. We show that this lack of smoothness only moderately affects the performance of our estimators. Taking $P = \text{Unif}([-1,1]^d)$ again, an example of a non-smooth Monge map is
\begin{align*}
    \textbf{(E2)} \, \, T(x) = \partial\pran{\frac{1}{2}\|x\|^2 + 2|x_1|} = x + 2e_1\text{sign}(x_1)\,. 
\end{align*}
Despite being the subgradient of a 1-strongly convex potential, \textbf{(E2)} is a discontinuous map that perturbs the input along the first coordinate. By approximating $\text{sign}(\cdot)$ with a smooth function, denoted $\text{sign}_\beta(\cdot)$\footnote{$\text{sign}_\beta(\cdot) =2(1+e^{-\beta x})^{-1} - 1$ with $\beta \gg 1$}, we can compare map estimation of \textbf{(E2)} and its smooth analogue:
\begin{align*}
    \textbf{(E2')} \, \, T(x) = x + 2e_1\text{sign}_\beta(x_1)\,.
\end{align*}
In \Cref{fig: nonsmooth_comp}, we see that the MSEs of both estimators are somewhat lower when the optimal map is smooth, with the Sinkhorn map still better.

\begin{figure}[h!]
    \centering
    \begin{subfigure}[h]{0.45\textwidth}
        \includegraphics[width=\textwidth]{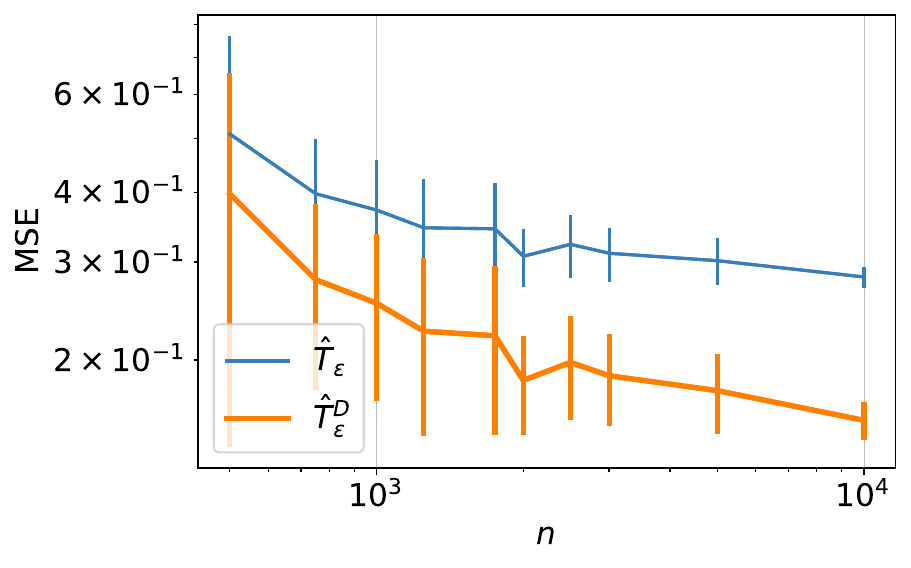}
        \caption{Example \textbf{(E2)} in $d=5$ with $\eps = 0.5$}
    \end{subfigure}
    \begin{subfigure}[h]{0.45\textwidth}
        \includegraphics[width=\textwidth]{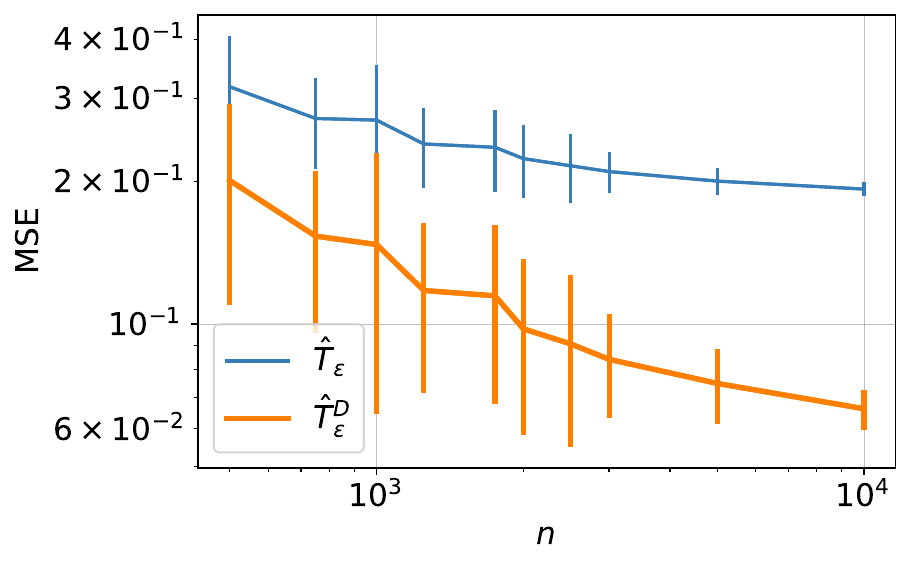}
        \caption{Example \textbf{(E2')} in $d=5$ with $\eps=0.5$ and $\beta=50$}
    \end{subfigure}
    \caption{Smoothing \textbf{(E2)} lowers the MSE for both \textcolor{blue_plots}{$\hat{T}_\eps$} and \textcolor{orange_plots}{$\hat{T}_\eps^D$}.}
    \label{fig: nonsmooth_comp}
\end{figure}
\subsection{Potential pitfalls in finite-sample estimation}\label{sec: finite_samp_main}
We return to the example of estimating the transport map between two Gaussian distributions. Recall the behavior predicted in  \cref{thm: err_g1} and \cref{thm: err_g2} and verified in \cref{fig:closedform_behavior}:  for $\eps$ sufficiently small, the debiased map provably gives a better approximation to the Monge map. However, \Cref{fig: g2g_finite_sample_concentrated} shows that the \emph{statistical} performance of the debiased map estimator can be substantially worse.

We fix the source distribution as $P = \cN(0,I_d)$ where $d$ is the dimension, and we randomly generate covariance matrices as outlined in \Cref{sec: g2g_mse_experiments} to create the target distribution $Q = \cN(0,\Sigma)$. We generate $\Sigma$ such that its eigenvalues are smaller than $1$, causing it to be more concentrated than the source.  Across 15 trials, we learn the map using $N$ points labeled in the figures. The lines labeled $N = \infty$ denote the error in the infinite-sample limit, obtained using the closed-form expressions given in §~\ref{sec:gauss}.

In $d=2$, we see the performance of $\hat{T}^D_\eps$ plateau at small values of $\eps$ even when using $N = 10^4$ points to learn the map, whereas $\hat{T}_{\eps}$ does not seem to experience this effect. For $d=15$, the performance of both estimators degrades, but the effects are worse for the Sinkhorn map, which is worse than the Entropic map for all $\eps \leq 1$.
We conjecture that the Sinkhorn map suffers more dramatically from finite-sample effects because the term $\nabla \hat \alpha_\eps(x)$ injects additional statistical noise into the map estimate.
\begin{figure}[h!]
    \centering
    \begin{subfigure}[h]{0.45\textwidth}
        \includegraphics[width=\textwidth]{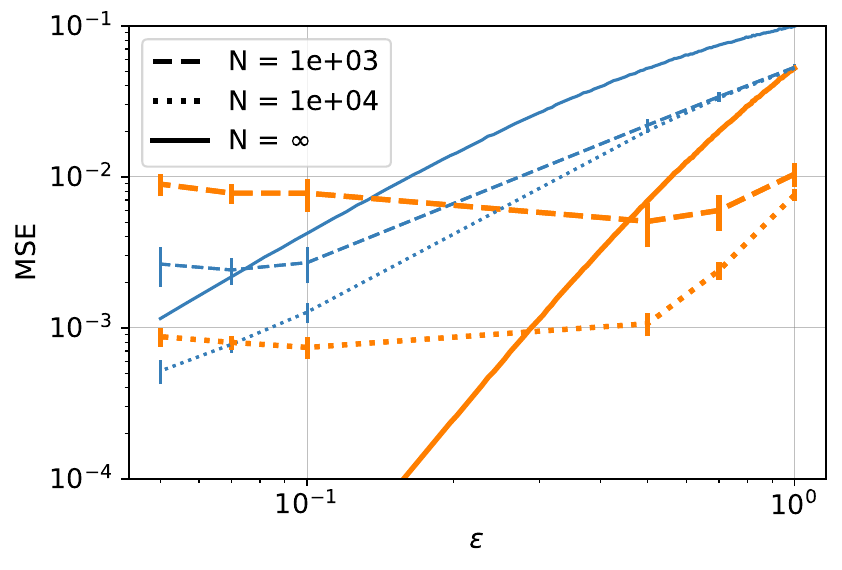}
        \caption{\textcolor{blue_plots}{$\hat{T}_\eps$} vs. \textcolor{orange_plots}{$\hat{T}_\eps^D$} with $\Sigma$ concentrated in $d=2$  }
    \end{subfigure}
    \begin{subfigure}[h]{0.45\textwidth}
        \includegraphics[width=\textwidth]{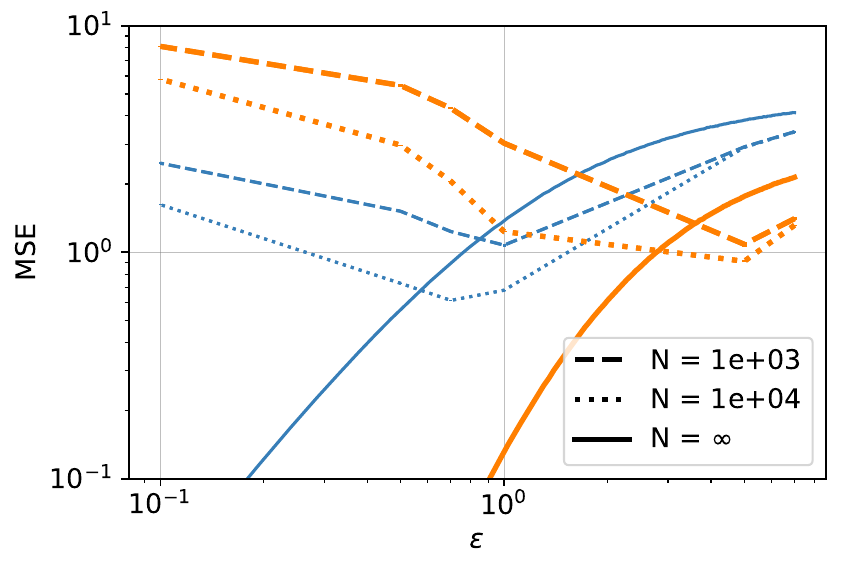}
        \caption{\textcolor{blue_plots}{$\hat{T}_\eps$} vs. \textcolor{orange_plots}{$\hat{T}_\eps^D$} with $\Sigma$ concentrated in $d=15$ }
    \end{subfigure}
    \caption{We observe strong finite-sample effects for the Sinkhorn map when the regularization parameter is small. While curse-of-dimensionality effects are pervasive in OT, Figure 5(a) shows that these effects occur even in $d=2$ for the Sinkhorn map.}
    \label{fig: g2g_finite_sample_concentrated}
\end{figure}

\subsection{Application: predicting trajectories of genomes}
We turn our attention to an application of map estimation using real-world data, where practitioners may not have a priori knowledge of a map even existing between the source and target measures. Such an example arises in  \citep{schiebinger2019optimal,moriel2021novosparc,Demetci2021.SCOTv2}, where the task is to infer cellular evolution from population measurements. The original data is temporal, where cell measurements are taken across 18 days, and each sampled data point consists of over 1000 gene expressions. Following the setup of \citet{schiebinger2019optimal}, we project the genes onto $\R^{30}$ using PCA. Finally, we normalize the data so that each datapoint lies in a ball of unit radius. We denote the source and target distributions by $\bm{X}$ and $\bm{Y}$, corresponding to the sampled data at \texttt{day0} and \texttt{day1}.

Across a range of $\eps$ values, we perform the following experiment across 20 trials where the train/test split is 50/50. For a fixed value of $\eps$, we use $(\bm{X}_{\text{train}},\bm{Y}_{\text{train}})$ to learn the mappings $\hat{T}_\eps$ and $\hat{T}^D_\eps$. Since there is no a priori notion of an optimal map, checking the MSE is not possible. Instead, we compute the (discrete) $W_2$ distance between a predicted mapping, such as ${\hat{T}}_\eps(\bm{X}_{\text{test}})$ (similarly for the Sinkhorn map), and $\bm{Y}_{\text{test}}$, newly sampled points from the target distribution.  We compute the (unregularized) $W_2$ distance using the Python OT (POT) package \citep{flamary2021pot}.

\begin{figure}[h!]
    \centering
    \includegraphics[width=0.45\textwidth]{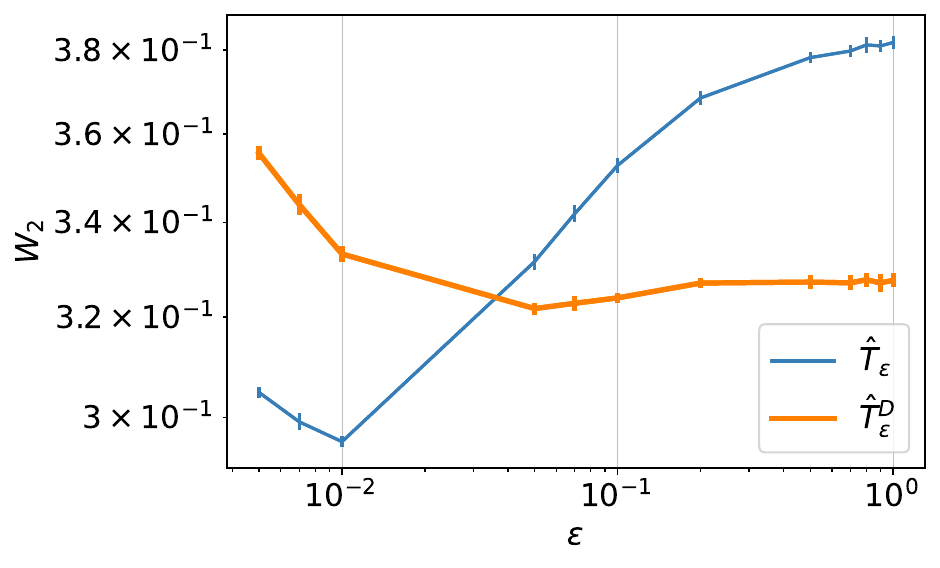}
    \caption{Comparing $W_2(\hat{T}(\bm{X}_{\text{test}}),\bm{Y}_{\text{test}})$ for different levels of $\eps$. As $\eps$ gets smaller, we see the same finite-sample effects from \Cref{sec: finite_samp_main}, and \textcolor{blue_plots}{$\hat{T}_\eps$} starts to perform better than \textcolor{orange_plots}{$\hat{T}^D_\eps$}.}
    \label{fig: genome_plot}
\end{figure}
\Cref{fig: genome_plot} shows the computed $W_2$ distance between the predicted mapping and test data. We notice that as $\eps$ get smaller, $\hat{T}_\eps$ begins to dominate. This is expected given our observations from \Cref{fig: g2g_finite_sample_concentrated} in the high-dimensional regime. As in \Cref{fig: g2g_finite_sample_concentrated}, the covariance of the target distribution \texttt{day1} has smaller eigenvalues than that of \texttt{day0}.

\Cref{fig: genome_viz} is a visualization of the first two components of the maps and the target distribution, which parallels \Cref{fig: viz_elliptic}: for large values of $\eps$, the Entropic map is heavily concentrated towards the mean of the target data. 
On the other hand, as $\ep \to 0$, the performance of $\hat T_\ep$ continues to improve, whereas that of $\hat T_\ep^D$ does not.

\begin{figure}[h!]
    \centering
    \includegraphics[width=0.45\textwidth]{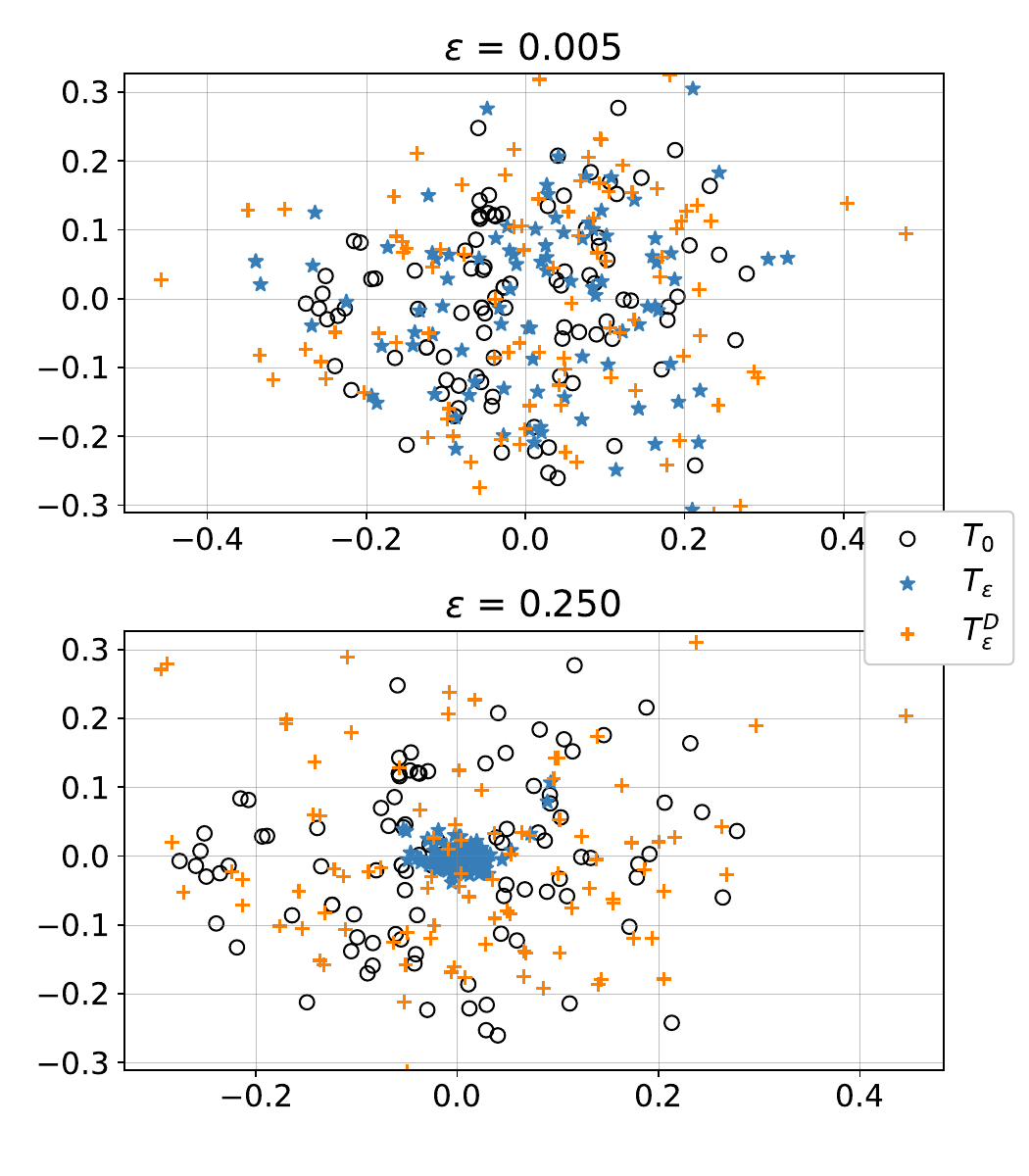}
    \caption{Visualizing genome predictions of the Entropic and Sinkhorn maps for small and large values of $\eps$ through the first two principal components. Much like \Cref{fig: viz_elliptic}, for large $\eps$, the Entropic map is biased towards the mean of the target.}
    \label{fig: genome_viz}
\end{figure}

\section{Conclusion}
In this work, we investigate theoretical and empirical properties of two estimators of Monge maps: the Entropic map and its debiased counterpart, the Sinkhorn map. In the population regime, we characterize their convergence as the regularization strength $\eps$ approaches either $0$ or $\infty$, and illustrate these results experimentally. Our findings indicate that debiasing does not always help, contradicting a dominant belief in the regularized 
OT literature. 
For example, in the Gaussian-to-Gaussian transportation problem, we prove (in the population setting) that the Sinkhorn map is a strictly better estimator as $\eps \to 0$. However, we also show, both theoretically and empirically, that for a \textit{fixed} $\eps$, the Sinkhorn map can perform worse than the Entropic map, and that these effects are magnified with small sample sizes. We confirm these findings experimentally on real datasets. On the other hand, when empirically estimating smooth maps, we notice that for an \textit{appropriately chosen} value of $\eps$, the Sinkhorn map can yield better estimates than the Entropic map. These findings indicate that the benefits of debiasing are not robust to the choice of the regularization parameter, and that it should be applied selectively.

\subsubsection*{Acknowledgements}
AAP and JNW acknowledge the support of a Google Research Collabs grant.

\bibliography{icmlsinkhornmonge}
\bibliographystyle{icml2022}

\newpage
\appendix
\section{Remaining (Entropic) OT facts}\label{sec: main_proofs_ot}
Let $\Lambda_\eps = (2\pi\eps)^{d/2}$. For two measures with bounded densities and compact support, $\OTep$ admits the following \textit{dynamical} formulation \citep{chizat2020faster,conforti2021formula}
\begin{align}\label{eq: dyn_otep}
\OTep(P,Q) + \eps\log(\Lambda_\eps) &= \inf_{\rho,v} \int_0^1 \int_{\R^d}  \pran{\frac 12 \|v(t,x)\|^2_2  + \frac{\eps^2}{8}\|\nabla_x \log(\rho(t,x))\|^2_2} \rho(t,x) \dd x \dd t \\
&- \frac{\eps}{2}(\Ent(P) + \Ent(Q) \nonumber,
\end{align}
subject to $\partial_t \rho + \nabla \cdot (\rho v) = 0$, called the \textit{continuity equation}, with $\rho(0,\cdot) = p(\cdot)$ and $\rho(1,\cdot) = q(\cdot)$. 

When $Q = P$, we can upper bound the left-hand side of \Cref{eq: dyn_otep} with the naive choices $v(t,x) \equiv 0$ and $\rho(t,\cdot) = p(\cdot)$, resulting in
\begin{align*}
    \OTep(P,P) + \ep\log(\Lambda_\eps) \leq \int_{\Rd} \frac{\eps^2}{8}\|\nabla \log(p(x))\|^2p(x) \dd x - \eps\Ent(P)\,,
\end{align*}
or equivalently,
\begin{align}\label{eq: otep_pp_bound}
    \OTep(P,P) + \eps\log(\Lambda_\eps) + \eps\Ent(P) \leq \frac{\eps^2}{8}I_0(P,P)\,. 
\end{align}
Finally, we recall the modified dual formulation of $\OTep(P,P)$ from \citet{pooladian2021entropic}
\begin{proposition}\label{prop: otpp_newdual}
Assume $P$ has finite second moment and let $\gamma_\eps$ be the optimal entropic plan from $P$ to itself. Then,
\begin{align}
    \OTep(P,P) = \sup_{\eta \in L^1(\gamma_\eps)} \int \eta \dd \gamma_\eps - \eps \iint e^{(\eta(x,y) - \frac{1}{2}\|x-y\|^2)/\eps}\dd P(x) \dd P(y) + \eps\,.
\end{align}
\end{proposition}
\section{Proofs from \Cref{sec:results}}\label{sec: main_proofs}

For a function $f : \R^d \to \R$, we define its \textit{convex conjugate} by $f^*(y) := \sup_{x}\{x^\top y - f(x)\}$. If $X \sim \cN(\mu,\Sigma)$, recall that for any $t \in \R^d$,
\begin{equation}\label{eq: mgf}
    \E[\exp\{t^\top X\}] = \exp\left\{t^\top\mu + \frac{1}{2}t^\top \Sigma t\right\}\,.
\end{equation}
\begin{proof}[Proof of \Cref{lem: alpha_eps_bound}]
Our proof technique is based on \citep{pooladian2021entropic,pal2019difference}. Let $q^x_\eps(y)$ denote the density of $\cN(x,\eps I_d)$.
Using \Cref{prop: otpp_newdual}, writing $\gamma_\eps$ as the optimal entropic plan from $P$ to itself, we plug in the test function
$$ \eta(x,y) = \eps(\chi(x,y) - \log(\Lambda_\eps) - \log(p(y)) ) $$
with $\chi(x,y) = h(x)^\top(y-x) - (\eps/2)\|h(x)\|^2$, where we omit the dependence on $h$ for the time being. This results in
\begin{align}
\OTep(P,P) &\geq \sup_\chi \iint \eps (\chi(x,y) - \log(\Lambda_\eps) - \log(p(y)) ) \dd \gamma_\eps(x,y) - \eps \iint  e^{\chi(x,y)} q^x_\eps(y) \dd y \dd P(x) + \eps\\
&= \sup_\chi \eps \iint \chi(x,y) \dd \gamma_\eps(x,y) - \eps \log(\Lambda_\eps) - \eps\Ent(P)  - \eps \iint  e^{\chi(x,y)} q^x_\eps(y) \dd y \dd P(x) + \eps
\end{align}
Rearranging, this yields
\begin{align}
\sup_\chi \iint \chi(x,y) \dd \gamma_\ep(x,y) - \iint \brac{e^{\chi(x,y)} - 1}q^x_\eps(y) \dd y \dd P(x) \leq \frac{1}{\eps}\pran{\OTep(P,P) + \ep\log(\Lambda_\ep) + \ep\Ent(P)} \,.
\end{align}
It follows by \Cref{eq: otep_pp_bound} that the right-hand side is bounded above by $\frac{\ep}{8}I_0(P,P)$. By well-known properties of moment generating functions of Gaussians (see \Cref{eq: mgf}), it also holds that
\begin{align}
\E_{Y \sim q_\eps^x}\brac{e^{v^\top(Y-x) - (\eps/2)\|v\|^2} - 1} = 0
\end{align}
for all $v \in \R^d$. Identifying $h(x)$  with $v$, we see that this implies 
\begin{align}
\iint \brac{e^{\chi(x,y)} - 1}q^x_\eps(y) \dd y \dd P(x) = \int \left\{ \E_{Y \sim q_\eps^x}\brac{e^{h(x)^\top(Y-x) - (\eps/2)\|h(x)\|^2} - 1}  \right\} \dd P(x) = 0\,.
\end{align}
Our bound now simplifies to
\begin{align}
\sup_h \iint h(x)^\top (y-x) - \frac{\eps}{2}\|h(x)\|^2 \dd \gamma_\ep - 0 \leq \frac{\eps}{8}I_0(P,P).
\end{align}
Taking $h(x) = -\eps^{-1}\nabla \alpha_\eps$, we use the fact that $\E_{\gamma_\ep}[Y|X=x] = x - \nabla \alpha_\eps(x)$ by definition, and that $\gamma_\eps \in \Pi(P,P)$, resulting in 
\begin{align}
\iint (-\eps^{-1}\nabla \alpha_\eps(x))^\top(y - x) - \frac{1}{2\eps}\|\nabla \alpha_\eps(x)\|^2 \dd \gamma_\eps  = \frac{1}{2\eps}\int \| \nabla \alpha_\eps(x)\|^2 \dd P(x) \leq \frac{\eps}{8}I_0(P,P).
\end{align}
Rearranging the remaining constants gives the desired bound.
\end{proof}

\begin{proof}[Proof of \Cref{thm: eps_d_small}]
Expanding the square directly and applying Cauchy-Schwarz gives
\begin{align*}
\cR(T^D_\eps) &= \| T_\eps^D - T_0\|^2_{L^2(P)} \\
&= \| (T_\eps - T_0) + \nabla \alpha_\eps \|^2_{L^2(P)} \\
&= \|T_\eps - T_0\|^2_{L^2(P)} + \|\nabla\alpha_\eps\|^2_{L^2(P)} + 2\langle T_\eps - T_0,\nabla\alpha_\eps \rangle_{L^2(P)} \\
&\leq \cR(T_\eps) + \|\nabla\alpha_\eps\|^2_{L^2(P)} + 2\sqrt{\cR(T_\eps)}\|\nabla\alpha_\eps\|_{L^2(P)} \\
&\leq \cR(T_\eps) + (\eps^2/4)I_0(P,P) + \eps I_0(P,P)\sqrt{\cR(T_\eps)} \\
&\leq \cR(T_\eps) + \frac{3\eps}{2}I_0(P,P)\sqrt{\cR(T_\eps)}\,,
\end{align*}
where the last inequality holds holds for $\eps$ small enough. The second claim follows by applying \Cref{prop: pooladian_main} and expanding the terms.
\end{proof}

\begin{proof}[Proof of \Cref{prop: map_conv}]
Since $\pi_\eps$ and $P\otimes Q$ have first marginal $P$, we begin by applying the chain-rule for the KL divergence:
\begin{align*}
\KL{\pi_\eps}{P\otimes Q} = \E_{X \sim P} \KL{\pi_\eps^X}{Q}\,,
\end{align*}
where we denote $\pi_\eps^X$ as the (random) conditional density of $\pi_\eps$ given $X$. Since $P$ and $Q$ are compactly supported, we can apply the $T_1$-transport inequality \citep{van2014probability}, 
\begin{align*}
\KL{\pi_\eps}{P\otimes Q} &= \E_{X \sim P} \KL{\pi_\eps^X}{Q} \gtrsim \E_{X \sim P}W_1^2(\pi_\eps^X,Q).
\end{align*}
We can further lower-bound this by the squared difference of the means of $\pi_\eps^X$ and $Q$,
\begin{align*}
\KL{\pi_\eps}{P\otimes Q} &\gtrsim \E_{X \sim P}W_1^2(\pi_\eps^X,Q) \geq \E_{X \sim P}\| \E_{\pi_\eps}[Y|X] - \mu_Q\|^2 .
\end{align*}
Taking the limit $\eps \to \infty$, we invoke \Cref{lem: kl_limit} and obtain the first claim:
\begin{align*}
\|T_\eps - \mu_Q\|^2_{L^2(P)} \lesssim \KL{\pi_\eps}{P\otimes Q} \to 0.
\end{align*}
By an identical argument, it follows that 
\begin{align*}
\| \text{Id} - \nabla \alpha_\eps - \mu_P \|^2_{L^2(P)} \to 0 \quad (\eps \to \infty)\,,
\end{align*}
and thus the second claim follows by an application of the  triangle inequality:
\begin{align*}
\| T_\eps^D - (\mu_Q - \mu_P + \text{Id})\|^2_{L^2(P)} &= \| (T_\eps - \mu_Q) + (\nabla\alpha_\eps - (- \mu_P + \text{Id}))\|^2_{L^2(P)} \\
&\lesssim \| T_\eps - \mu_Q \|^2_{L^2(P)}+ \| \nabla\alpha_\eps + \mu_P - \text{Id}\|^2_{L^2(P)} \\
& \to 0
\end{align*}
\end{proof}

\begin{proof}[Proof of \Cref{thm: err_diff_main}]
By an application of the (reverse) triangle inequality,
\begin{align*}
\left| \|T_\eps - T_0\|_{L^2(P)} - \sqrt{\mathrm{Var}(Q)} \right| &= \left| \|T_\eps - T_0\|_{L^2(P)} - \|\mu_Q - Y\|_{L^2({Q})} \right| \\
&=  \left| \|T_\eps - T_0\|_{L^2(P)} - \|\mu_Q - T_0\|_{L^2({P})} \right| \\
&\leq \| T_\eps - \mu_Q \|_{L^2(P)}\,,
\end{align*}
which converges to zero by \Cref{prop: map_conv} as $\eps \to \infty$. Taking squares, we have the first claim; the second follows by a similar argument.
\end{proof}

\begin{proof}[Proof of \Cref{thm: counter_ex}]
Fix $\eps \in (0,1)$ and take $P = \cN(0,1)$ and $Q = \cN(0,\sigma^2)$. Set $\sigma^2 = \eps^{2m}$ for $m>0$ to be specified. Then
\begin{align*}
&\cR(T_\eps) = ((\eps^{2m} + (\eps^2/4))^{0.5} - (\eps/2) + \eps^m)^2, \\ 
&\cR(T_\eps^D) = ((\eps^{2m} + \eps^2/4)^{0.5} + 1 - \eps^m - (1+\eps^2/4)^{0.5})^2\,.
\end{align*}
Taking $m\to\infty$, we see that $\cR(T_\eps) \to 0$ and $\cR(T^D_\eps) \to (\eps/2) + 1 - (1+\eps^2/4)^{0.5} > 0$. Thus, for any $M > 0$, there exists an $m$ large enough that $\cR(T^D_\eps) \geq M \cdot \cR(T_\eps)$.
\end{proof}

\section{Proofs from \Cref{sec:gauss}}\label{sec: gauss_proofs}

\begin{proof}[Proof of \Cref{prop: teps_g}]
The optimal entropic plan between two Gaussians is known to be the following multivariate Gaussian distribution \citep{janati2020entropic}
\begin{align}\label{eq: pieps_g2g}
\pi_\eps = N\left( \begin{pmatrix}
0 \\
b
\end{pmatrix},
\begin{pmatrix}
A  & \Sigma_\eps \\
\Sigma_\eps^\top & B 
\end{pmatrix}
 \right),
\end{align} 
where $\Sigma_\eps = A^{1/2}(A^{1/2}BA^{1/2} + (\eps^2/4)I)^{1/2}A^{-1/2} - (\eps/2)I.$ The conditional mean of this joint multivariate Gaussian is known to have the following closed form solution \citep{petersen2008matrix}, which completes the claim
\begin{align*}
\E_{\pi_\eps}[Y|X=x] = \Sigma_\eps^\top A^{-1}x + b = C_\eps^{AB}x+b.
\end{align*}
\end{proof}

\begin{proof}[Proof of \Cref{prop: alpha_g}]
If $\gamma_\eps$ is the optimal entropic plan arising from $\OTep(P,P)$, mimicking \Cref{eq: pieps_g2g}, we have
\begin{align*}
\gamma_\eps = N\left( \begin{pmatrix}
0 \\
0
\end{pmatrix},
\begin{pmatrix}
A  & \tilde{\Sigma}_\eps \\
\tilde{\Sigma}_\eps^\top & A 
\end{pmatrix}
 \right)\,,
\end{align*}
where $\tilde{\Sigma}_\eps = A^{1/2}(A^2 + (\eps^2/4)I)^{1/2}A^{-1/2}-(\eps/2)I$. The conditional expectation in this case simplifies to
\begin{align*}
    \E_{\pi_\eps^A}[Y|X=x] = \tilde{\Sigma}_\eps^\top A^{-1}x = C_\eps^{AA}x.
\end{align*}
By combining this computation with \cref{eq: t_eps_expectation} and \cref{eq: sinkhorn_map}, we have that
\begin{align*}
    \nabla \alpha_\eps(x) = x - (C_\eps^{AA}x) = (I-C_{\eps}^{AA})x\,.
\end{align*}
\end{proof}

\begin{proof}[Proof of \Cref{thm: err_g1}]
We begin by creating short-hand notation for convenience. Writing $M = A^{1/2}BA^{1/2}$, then $C_0^{AB} = A^{-1/2}M^{1/2}A^{-1/2}$. We also denote $M_\eps := (M+(\eps^2/4)I)^{1/2} - M^{1/2},$ which is positive semidefinite. A direct computation yields (cf. \cite{petersen2008matrix})
\begin{align}
    \|T_\eps - T_0\|_{L^2(P)}^2 &= \E_{X \sim P}\|DX\|^2 = \E_{X\sim P}X^\top D^2 X = \Tr(D^2A),  
\end{align}
where $D := C_\eps^{AB} - C_0^{AB} = A^{-1/2}M_\eps A^{-1/2} - (\eps/2)A^{-1}$. Computing the square and post-multiplying by $A$ gives
\begin{align*}
    D^2A &= \frac{\eps^2}{4}A^{-2}A + (A^{-1/2}M_\eps A^{-1/2})^2A - \frac{\eps}{2}[A^{-1/2}M_\eps A^{-1/2}A^{-1} + A^{-1}A^{-1/2}M_\eps A^{-1/2}]A \\
    &= \frac{\eps^2}{4}A^{-1} + (A^{-1/2}M_\eps A^{-1/2})^2A - \frac{\eps}{2}[A^{-1/2}M_\eps A^{-1/2} + A^{-1}A^{-1/2}M_\eps A^{-1/2}A].
\end{align*}
Since the trace operator is linear and invariant under permutations, we arrive at
\begin{align*}
    \Tr(D^2A) &= \frac{\eps^2}{4}\Tr(A^{-1}) + \Tr(A^{-1/2}M_\eps A^{-1/2}A^{-1/2}M_\eps A^{-1/2}A) - \eps\Tr(A^{-1/2}M_\eps A^{-1/2}) \\
    &= \frac{\eps^2}{4}\Tr(A^{-1}) + \Tr(A^{-1/2}(M_\eps^2 - \eps M_\eps)A^{-1/2}) \\
    &\leq \frac{\eps^2}{4}\Tr(A^{-1}) + \Tr(M_\eps^2A^{-1}).
\end{align*}
Invoking \Cref{lem: sqrt_exp}, we arrive at
\begin{align*}
    \|T_\eps - T_0\|^2_{L^2(P)} \leq \frac{\eps^2}{4}\Tr(A^{-1}) + \eps^4\Tr(A^{-1})C_{\lambda_{\min}(M)},,
\end{align*}
but a direct computation shows that
\begin{align*}
    I_0(P,P) &= \int \|\nabla \log p(x))\|^2 p(x) \dd x \\
    &= \int \|A^{-1}x\|^2 p(x) \dd x \\
    &=  \E_{X \sim \cN(0,A)}[X^\top A^{-2}X]\\
    &= \Tr(A^{-2}A) \\
    &= \Tr(A^{-1})\,,
\end{align*}
which completes the proof.
\end{proof}

\begin{proof}[Proof of \Cref{thm: err_g2}]
We keep the same definitions of $M$ and $M_\eps$ as in \Cref{thm: err_g1}, and also introduce 
\begin{align*}
A_\eps := (A^2+(\eps^2/4)I)^{1/2} - A\,.
\end{align*}
As last time, we have by direct computation that 
\begin{align*}
    \| T_\eps^D - T_0\|^2_{L^2(P)} = \Tr(\tilde{D}^2A)\,
\end{align*}
where
\begin{align*}
    \tilde{D} := \tilde{C}_\eps^{AB} - C_0^{AB} = A^{-1/2}\pran{M_\eps - A_\eps}A^{-1/2}\,,
\end{align*}
where both $A_\eps, M_\eps \succeq 0$. We begin by computing the squared term inside the trace:
\begin{align*}
    (\tilde{C}_\eps^{AB} - C_0^{AB})^2 &= (A^{-1/2}M_\eps A^{-1/2})^2 + (A^{-1/2}A_\eps A^{-1/2})^2 - A^{-1/2}M_\eps A^{-1} A_\eps A^{-1/2} - A^{-1/2}A_\eps A^{-1}M_\eps A^{-1/2}.
\end{align*}
By the permutation invariance of trace, we arrive at
\begin{align*}
    \Tr(\tilde{D}^2A) &= \Tr((M_\eps^2 + A_\eps^2 - 2A_\eps M_\eps)A^{-1}).
\end{align*}
Again, by \Cref{lem: sqrt_exp}, we have
\begin{align*}
    M_\eps = \frac{\eps^2}{8}M^{-1/2} + O_{A,B}(\eps^4), \, A_\eps = \frac{\eps^2}{8}A^{-1} + O_A(\eps^4)\,.
\end{align*}
Taking squares and expanding within the trace operator yields the desired claim
\begin{align*}
    \Tr((M^2_\eps + A_\eps^2 - 2A_\eps M_\eps)A^{-1}) = \frac{\eps^4}{64}\Tr(M^{-1}A^{-1} + A^{-3} - 2M^{-1/2}A^{-2}) + O_{A,B}(\eps^6)
\end{align*}
\end{proof}

\section{Lemmas}\label{sec:lemmas}
\begin{lemma}\label{lem: sqrt_exp}
For $C \succ 0$, the following expansion holds
\begin{equation*}
    (C + \eta I)^{1/2} = C^{1/2} + \sum_{k\geq 1}(-1)^{k-1}\eta^kC^{1/2-k}c_k,
\end{equation*}
with $c_k = 2^{-k}\cdot(2k-1)!!$. In particular, taking $\eta = \eps^2/4$, we have that the expansion is, to second order
\begin{align}
    (C + (\eps^2/4)I)^{1/2} = C^{1/2} + \frac{\eps^2}{8}C^{-1/2} - \frac{\eps^4}{128}C^{-3/2} + R_\eps\,,
\end{align}
where the remainder term $R_\eps$ satisfies $\|R_\eps\| \lesssim \lambda^{-5/2}_{\min}(C) \eps^6$.
\end{lemma}
\begin{proof}
Since $I$ commutes with $C \succ 0$, the above claim follows by considering the Taylor expansion of the scalar function $f(x) = (x+\eta)^{1/2}$ and replacing the relevant quantities with their matrix counterparts. The second claim follows by noticing that the second order term is dominated by the smallest eigenvalue of $C$ raised to the relevant power, which amounts to $\lambda_{\min}^{-5/2}(C)I$, 
\end{proof}


\begin{lemma}\label{lem: kl_limit}
Let $\pi_\eps$ be the minimizer for $\OTep(P,Q)$ under with $P$ and $Q$ having compact support. Then $\KL{\pi_\eps}{ P \otimes Q} \to 0$ as $\eps \to \infty$.
\end{lemma}
\begin{proof}[Proof of \Cref{lem: kl_limit}]
Note that $\OTep(P,Q)$ is bounded above by $C^* := \iint \frac{1}{2}\|x-y\|^2 \dd (P\otimes Q) > 0$, so
$$ \iint \frac{1}{2}\|x-y\|^2 \dd \pi_\eps + \eps \KL{\pi_\eps}{ P\otimes Q} \leq C^*.$$
The first term on the left-hand side is also bounded above by some constant, as $\pi_\eps \in \Gamma(P,Q)$, where $P$ and $Q$ have compact support. Taking the limit of $\eps \to \infty$, it must be the case that $\KL{\pi_\eps}{P\otimes Q} \to 0$.
\end{proof}
\vfill
\pagebreak
\section{Numerics}\label{sec: app_numerics}
\subsection{Compact, elliptically contoured distributions}\label{sec: ellip_distr}
Here we recall some facts about how to sample from compact, elliptically contoured distributions, as outlined in \cite{chizat2020faster}. To sample from such a (centered) distribution, one can use the following simple recipe: for a given covariance matrix $A \in \R^d$, 
\begin{enumerate}
    \item Sample $U \sim \text{Unif}(\mathbb{S}^{d-1})$ (Uniform distribution over the $d-1$ sphere),
    \item Sample $Z \sim \cN(0,1)$,
    \item Set $R := \alpha|\arctan{Z/\beta}|^{1/d}$, where $\alpha > 0$ such that $\E[R^2]=d$,
    \item $X = R \cdot A^{1/2}U$.
\end{enumerate}
In Step 3, we use the choice of $\beta=2$ and $R$ is ultimately computed via Monte Carlo integration with $10^7$ points. For simplicity, we choose the same covariance matrices as in \cite{chizat2020faster} as their code is publicly available.
\subsection{Synthetic examples}
We run a log-stabilized version of Sinkhorn's algorithm to alleviate numerical instabilities that would otherwise arise for small choices of $\eps$. For both $d=5$ and $d=10$, we have $N_S = 50000$ as the number of points to approximate the MSE $\|\hat{T} - T_0\|^2_{L^2(P)}$ via Monte-Carlo integration after having learned the maps. For fixed $\eps$ (we chose $\eps = 0.5$ in both $d=5$ and $d=10$), we sample $n$ points from the source (where $n$ varies between $100$ and $10000$) and map the points, effectively generating samples from the target distribution. We run this procedure 20 times to generate error bars on the plots.

Example \textbf{(E3)} is between two elliptically contoured distributions, where the optimal transport map is known to be a linear map. Specifically, for centered, compact, elliptically contoured distributions with covariance $A$ for the source and $B$ for the target, then the optimal transport map is written as
$$ T(x) = A^{1/2}(A^{-1/2}BA^{-1/2})^{1/2}A^{1/2}x.$$

Example \textbf{(E4)} is the coordinate-wise exponential map, $T(x) = (\exp(x_i))_{i=1}^d$, with source distribution $P = \text{Unif}([-1,1]^d)$.

\subsubsection*{Remaining plots in 5D}

\begin{figure}[h]
    \centering
    \begin{subfigure}[h]{0.4\textwidth}
        \includegraphics[width=\textwidth]{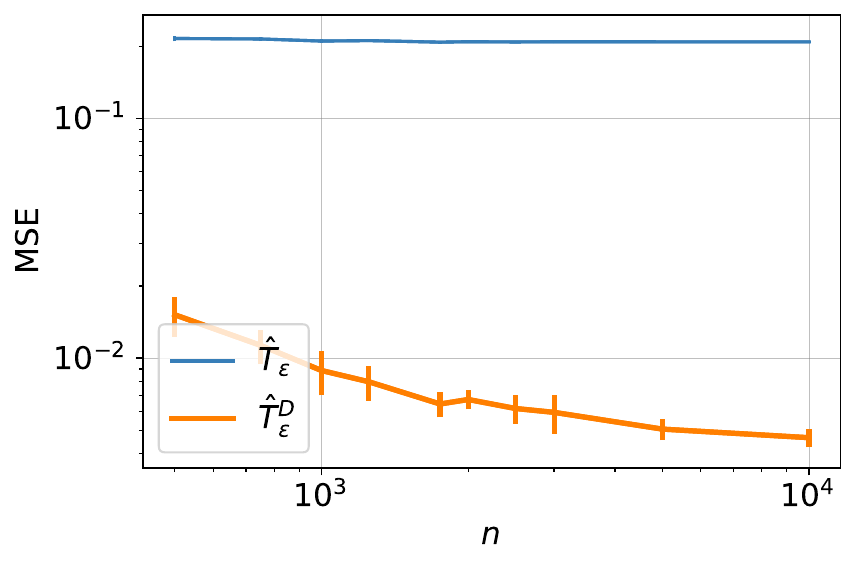}
        \caption{Example \textbf{(E3)}}
    \end{subfigure}
    \begin{subfigure}[h]{0.4\textwidth}
        \includegraphics[width=\textwidth]{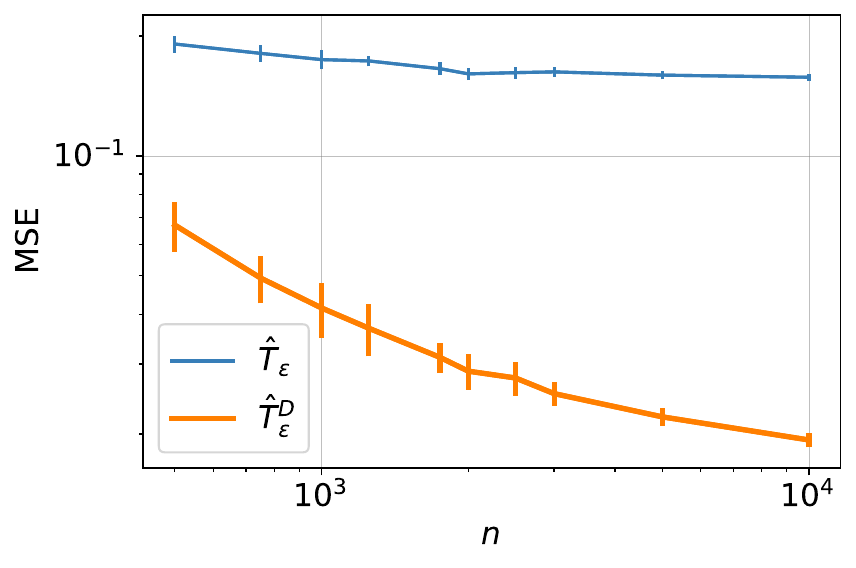}
        \caption{Example \textbf{(E4)}}
    \end{subfigure}
    \label{fig: d5_examples_remaining}
    \caption{Remaining examples in $d=5$}
\end{figure}
\pagebreak
\subsubsection*{Examples plots in 10D}

\begin{figure}[h]
    \centering
    \begin{subfigure}[h]{0.45\textwidth}
        \includegraphics[width=\textwidth]{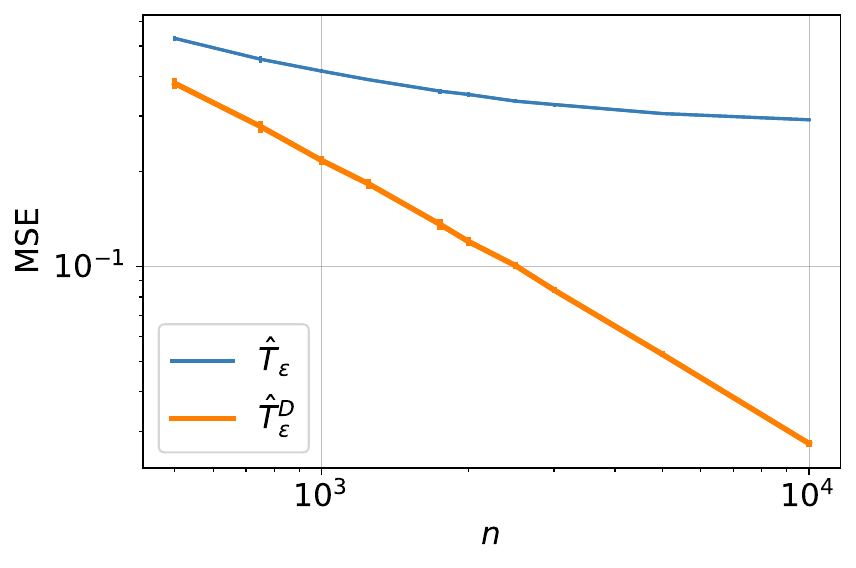}
        \caption{Example \textbf{(E1)}}
    \end{subfigure}
    \begin{subfigure}[h]{0.45\textwidth}
        \includegraphics[width=\textwidth]{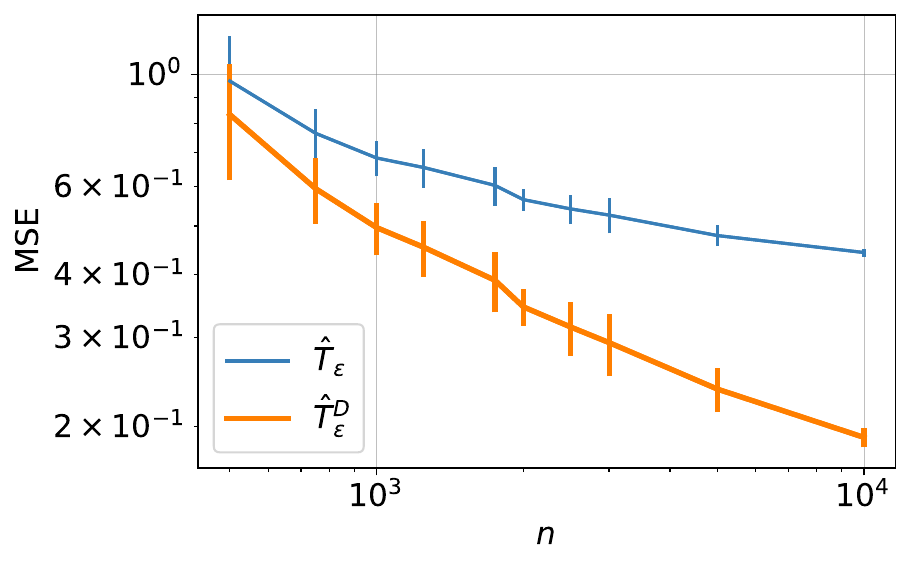}
        \caption{Example \textbf{(E2)}}
    \end{subfigure}
    \begin{subfigure}[h]{0.45\textwidth}
        \includegraphics[width=\textwidth]{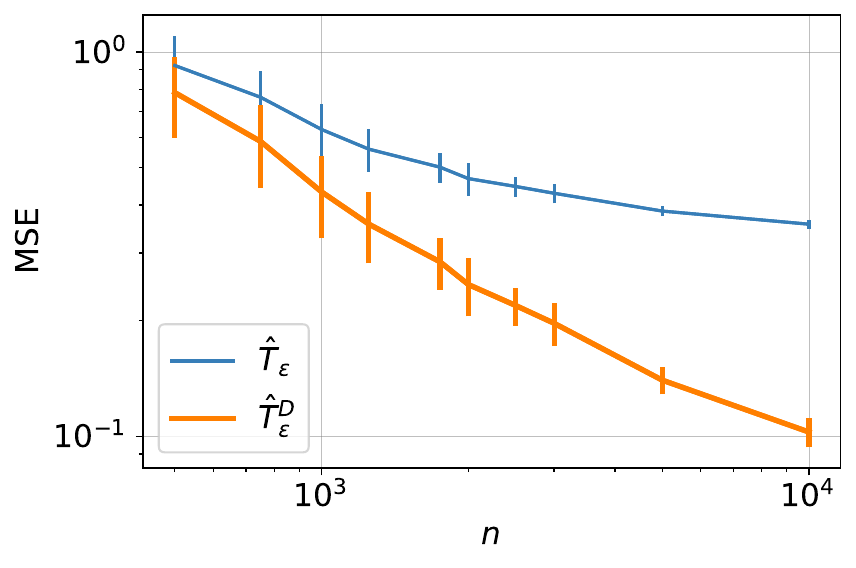}
        \caption{Example \textbf{(E2')}}
    \end{subfigure}
    \begin{subfigure}[h]{0.45\textwidth}
        \includegraphics[width=\textwidth]{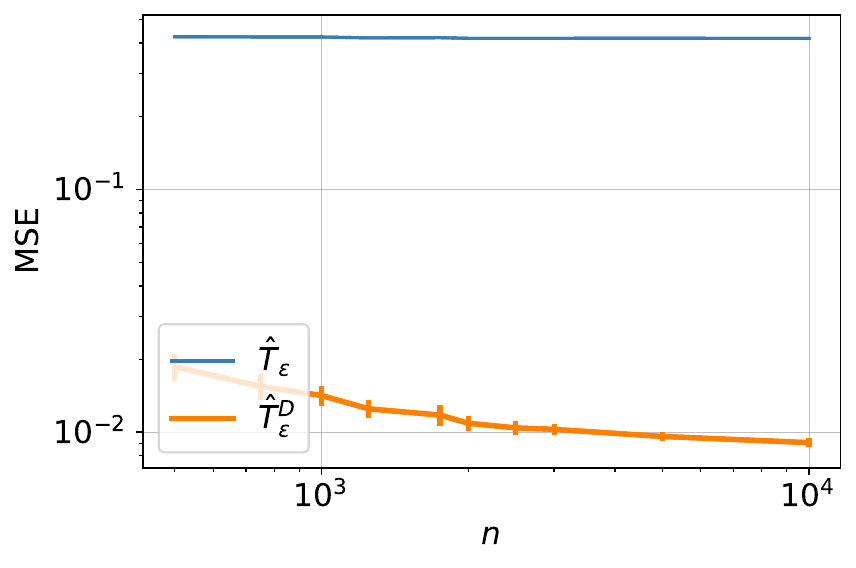}
        \caption{Example \textbf{(E3)}}
    \end{subfigure}
    \begin{subfigure}[h]{0.45\textwidth}
        \includegraphics[width=\textwidth]{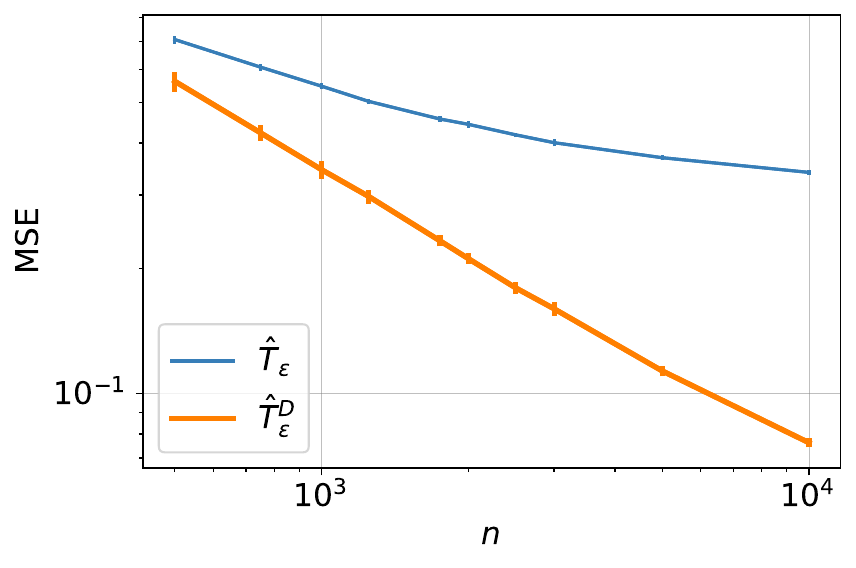}
        \caption{Example \textbf{(E4)}}
    \end{subfigure}
    \caption{Previous examples but with $d=10$ }
    \label{fig: d10_examples}
\end{figure}

\pagebreak
\subsection{Gaussian to Gaussian map estimation}\label{sec: g2g_mse_experiments}
\subsubsection{Randomly generating covariance matrices}
We follow a procedure similar to what is outlined in \cite{chizat2020faster}, which we re-write here for completeness. We begin with a matrix $M \in \R^{d\times k}$ with $M_{ij} \sim N(0,1)$ i.i.d. entries, and $k=d/\alpha$ for $\alpha\in(0,1)$. Defining $\Tilde{A} = MM^\top \succeq 0$, it is known that the eigenvalues of $\Tilde{A}$ are contained (with high probability) a small enlargement of the interval $[(1 -\sqrt{\alpha})^2,(1+\sqrt{\alpha})^2 ]$; this follows from Gordon's inequality~\citep[Theorem 5.32]{MR2963170}. Choosing $\alpha=1/3$ and writing $A = \gamma A/\Tr(\Tilde{A})$ for a positive constant $\gamma$, we can (randomly) control the spread of the eigenvalues of the covariance matrix $A \in \R^{d\times d}$ through the parameter $\gamma > 0$.

For the ``concentrated" covariance matrices, we chose $\gamma=0.2$ for $d=2$ and $\gamma=5$ in $d=15$. For the ``spread out" covariance matrices, we chose $\gamma=5$ for $d=2$ and $\gamma=20$ in $d=15$.

\subsubsection{Remaining plots}
In all examples here, $P = \cN(0,I_d)$ and $Q = \cN(0,\Sigma)$, where $\Sigma$ is a randomly generated covariance matrix that has some eigenvalues greater than 1, and some less than 1, which we call ``spread out" (see above). We present examples in $d=2$ and $d=15$, illustrating finite-sample effects in both low and high dimensions. We perform the MSE estimates using Monte-Carlo integration with $5\cdot 10^5$ samples. For each figure, we perform 15 random trials of map estimation where we learn the map with the value $N$ in the plots, and vary $\eps$.
\begin{figure}[h]
    \centering
    \begin{subfigure}[h]{0.45\textwidth}
        \includegraphics[width=\textwidth]{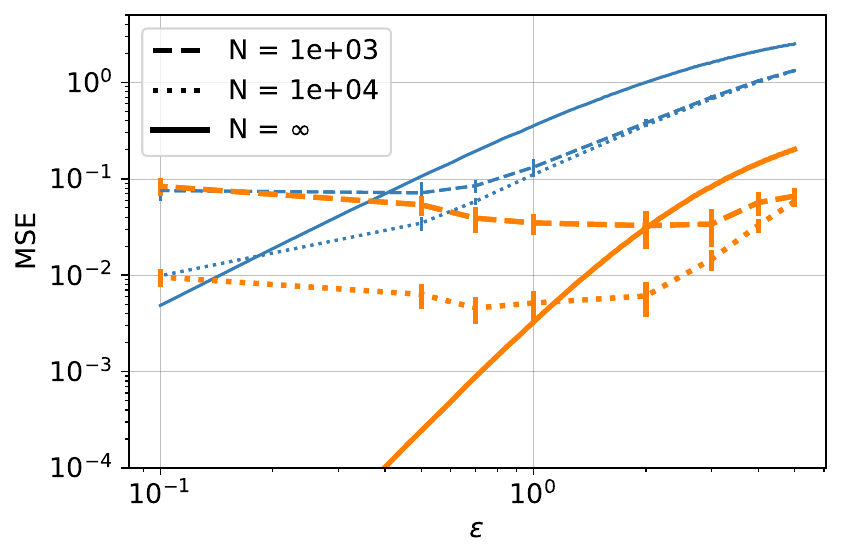}
        \caption{ $d=2$ with $\Sigma$ spread out }
    \end{subfigure}
    \begin{subfigure}[h]{0.45\textwidth}
        \includegraphics[width=\textwidth]{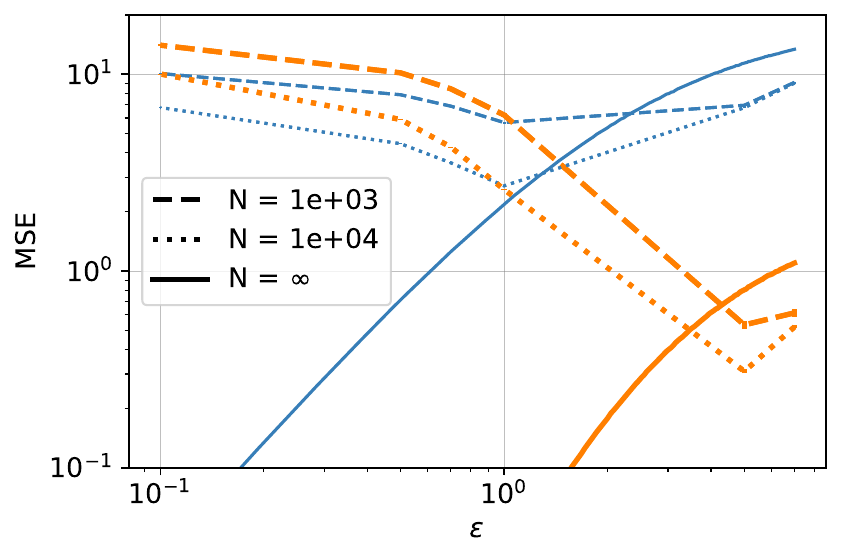}
        \caption{ $d=15$ with $\Sigma$ spread out }
    \end{subfigure}

    \caption{Finite-sample effect for \textcolor{blue_plots}{$\hat{T}_\eps$} vs. \textcolor{orange_plots}{$\hat{T}_\eps^D$} when $\Sigma$ is spread out }
    \label{fig: g2g_examples}
\end{figure}
\end{document}